\documentclass[11pt]{amsart}
\usepackage{rotating}
\usepackage{mathcomp,amscd}
\usepackage{amssymb}
\usepackage{hyperref}
\usepackage{times}
\usepackage{enumerate}
\usepackage{color}

\usepackage[OT2,T1]{fontenc}
\DeclareSymbolFont{cyrletters}{OT2}{wncyr}{m}{n}
\DeclareFontFamily{OT1}{rsfs}{}
\DeclareFontEncoding{OT2}{}{} % to enable usage of cyrillic fonts
  
     \DeclareFontShape{OT1}{rsfs}{n}{it}{<-> rsfs10}{}
\DeclareMathAlphabet{\mathscr}{OT1}{rsfs}{n}{it}

\newcommand{\SL}{\mathrm{SL}}

\numberwithin{equation}{section}
\numberwithin{table}{section}
\numberwithin{figure}{section}
\newtheorem{Theorem}{Theorem}[section]
\newtheorem{Proposition}{Proposition}[section]
\newtheorem{Lemma}{Lemma}[section]

\newtheorem*{Question*}{Question}

\theoremstyle{remark}
\newtheorem*{Remark}{Remark}
\allowdisplaybreaks
    \title{Random lattice vectors in a set of size $O(n)$}

\usepackage[foot]{amsaddr}

\makeatletter
\renewcommand{\email}[2][]{%
  \ifx\emails\@empty\relax\else{\g@addto@macro\emails{,\space}}\fi%
  \@ifnotempty{#1}{\g@addto@macro\emails{\textrm{(#1)}\space}}%
  \g@addto@macro\emails{#2}%
}
\makeatother

\author{Seungki Kim$^\dag$}
\address[\dag]{Department of Mathematics, Korea Institute of Advanced Study}
\email{seungki.math@gmail.com}

\begin{document}

\begin{abstract}
We adopt the sieve ideas of Schmidt \cite{Schmidt1} \cite{Schmidt2} and S\"odergren \cite{Sangle} in order to study the statistics of vectors of a random lattice of dimension $n$ contained in a set of volume $O(n)$. We also give some sporadic applications of our results to number theory.
\end{abstract}

\maketitle

%%% introduction
\section{introduction}

We define a lattice to be a discrete cocompact subgroup of $\mathbb{R}^n$ with vector addition. The set of all $n$-dimensional lattices of covolume (or determinant) 1 is modeled as $X_n := \SL(n,\mathbb{Z}) \backslash \SL(n, \mathbb{R})$: the coset $\SL(n, \mathbb{Z})g$ corresponds to the lattice $\mathbb{Z}^n g \subseteq \mathbb{R}^n$. By Siegel \cite{Siegel} there exists a probability measure $\mu_n$ on $X_n$, inherited from the Haar measure of $\SL(n,\mathbb{R})$, that is invariant under the action of $\SL(n,\mathbb{R})$ on $X_n$ by right multiplication. The pair $(X_n, \mu_n)$ provides the standard notion of a \emph{random lattice}.

The purpose of the present paper is to discuss the following
\begin{Question*}
For $V \in \mathbb{R}_{\geq 0}$, denote by $B(V) \subseteq \mathbb{R}^n$ the open ball of volume $V$ centered at the origin. If $L$ is a random lattice, how is $L \cap B(V)$ distributed, in particular in the limit as $n \rightarrow \infty$?
\end{Question*}

If $V$ is fixed relative to $n$, many results have been established which suggest that $L \cap B(V)$ looks ``random'' in a certain sense. For instance, if $L$ is a random lattice in the sense of $\mu_n$,

\begin{Theorem}[Rogers \cite{Rogers}] \label{thm_intro_rogers}
Let $V \in \mathbb{R}_{\geq 0}$ be fixed. Then as $n \rightarrow \infty$, the distribution of $\frac{1}{2}|(L \backslash \{0\}) \cap B(V)|$, the number of nonzero vectors of $L$ up to sign that are contained in $B(V)$, converges weakly to the Poisson distribution of mean $V/2$.
\end{Theorem}

\begin{Theorem}[S\"odergren \cite{Sangle}] \label{thm_intro_s1}
%the commented-out version has a more precise description about the angles:
%Let $k \in \mathbb{Z}_{> 0}$ be fixed. Define $\nu_i(L) := U\|x_i(L)\|^n$, where $U$ equals the volume of the $n$-dimensional unit ball, and $x_i(L)$ is the $i$-th shortest nonzero vector (up to sign) of $L$. Furthermore, for $1 \leq i < j \leq k$ let $\alpha_{ij}(L) \in [0, \pi)$ denote the angle (take the one less than $\pi$) formed by $x_i(L)$ and $x_j(L)$.
%As $n \rightarrow \infty$, the random variable $(\nu_1(L), \ldots, \nu_k(L), \{\alpha_{ij}(L)\}_{1 \leq i < j \leq k})$ converges to the joint distributions of the first $k$ arrival times of a Poisson process with intensity $\frac{1}{2}$ and the distribution of the mutual angles of $k$ vectors uniformly chosen from $S^{n-1}$.

Let $k \in \mathbb{Z}_{> 0}$ be fixed. Define $\nu_i(L) := U\|x_i(L)\|^n$, where $U$ equals the volume of the $n$-dimensional unit ball, and $x_i(L)$ is the $i$-th shortest nonzero vector (up to sign) of $L$. Furthermore, for $1 \leq i \leq k$ let $u_i(L) = \frac{x_i(L)}{\|x_i(L)\|} \in S^{n-1} = \{x \in \mathbb{R}^n: \|x\| = 1\}$.

As $n \rightarrow \infty$, the random variable $(\nu_1(L), \ldots, \nu_k(L), u_1(L), \ldots, u_k(L))$ converges to the joint distributions of the first $k$ arrival times of a Poisson process with intensity $\frac{1}{2}$ and the distribution of the $k$ vectors (up to sign) uniformly chosen from $S^{n-1}$.
\end{Theorem}

\begin{Theorem}[S\"odergren \cite{Sangle}] \label{thm_intro_s2}
Continue with the notations of Theorem \ref{thm_intro_s1}. Let $\lambda_i(L)$ denote the $i$-th successive minimum of $L$. Then
\begin{equation*}
\mathrm{Prob}_{\mu_n}(\|x_i(L)\| = \lambda_i(L), \forall i = 1, \ldots, k) \rightarrow 1,
\end{equation*}
as $n \rightarrow \infty$. In particular, Theorem \ref{thm_intro_s1} would also hold if we defined instead $\nu_i(L) := U \lambda_i(L)^n$.
\end{Theorem}
(In the latter two theorems cited above, $k$ being fixed relative to $n$ is morally equivalent to $V$ being fixed, because by Theorem \ref{thm_intro_rogers} there are on average $V/2$ vectors modulo signs inside $B(V)$, and the standard deviation is about the square root of the mean.)

It is natural to ask whether $L \cap B(V)$ is still ``random'' in the sense of the cited theorems as $V$ or $k$ grows with $n$. This question has been first investigated in Kim \cite{Kim}, where the author extended Theorem \ref{thm_intro_rogers} to the case $V = o(\sqrt{n})$.

In the present paper, we extend Theorems \ref{thm_intro_rogers} and \ref{thm_intro_s2} to the case $V, k < Cn$, where $C$ is a tiny absolute constant --- see Propositions \ref{prop_lpoisson} and \ref{prop_s=m}, respectively, for precise statements. We prove a few more statements in the same spirit also in the case $V, k < Cn$, e.g. Proposition \ref{prop_indep}. We generalize Theorem \ref{thm_intro_s1} as well, though unfortunately we practically have only $k = o(n/\log n)$:

%main thm
%\begin{Theorem} \label{thm_star}
%Choose an $o_n(1)$ function $f(n)$. Take any sufficiently large $n$, and choose $d \leq n / \log n$. Choose $d$ disjoint annuli $A_1, \ldots, A_d \subseteq \mathbb{R}^n$, each of volume $2f(n)$. Also choose nonnegative integers $a_1, \ldots, a_d$ such that $N := a_1 + \ldots + a_d < Cn$, where $C > 0$ is some absolute constant, and choose a subset $T \in [0,\pi)^{\binom{N}{2}}$ obeying a certain symmetric condition (see the paragraphs below \eqref{eq_asep} for details).

%Denote by $x_1(L)$ the shortest nonzero vector modulo signs in $L \cap \bigcup A_i$, by $x_2(L)$ the next shortest vector, and so on. Let $\alpha_{ij}(L) \in [0, \pi)$ be the angle between $x_i(L)$ and $x_j(L)$. Then there exists an absolute constant $c > 0$ such that
%\begin{align} \label{eq_star}
%&\mathrm{Prob}_{\mu_n}(\mbox{$\frac{1}{2}|L \backslash \{0\} \cap A_i| = a_i$, and $(\alpha_{ij}(L))_{1 \leq i < j \leq N} \subseteq T$}) \\
%&= A(T)\prod_{i=1}^d p(f(n), a_i) + O(e^{-cn}), \notag
%\end{align}
%where $A(T)$ equals the probability that $N$ vectors (modulo sign) chosen uniformly from $S^{n-1}$ have their mutual angles as dictated by $T$, and $p(\lambda, x)$ is the probability mass function of the Poisson distribution of mean $\lambda$.
%\end{Theorem}

%main thm, a different phrasing
\begin{Theorem} \label{thm_star}
Again continue with the notations of Theorem \ref{thm_intro_s1}. Consider the following two point processes on $\mathcal{R}^n := (\mathbb{R}^n \backslash \{0\})/\{\pm 1\} \cong \mathbb{R}_{>0} \times S^{n-1}/\{\pm 1\}$:
\begin{itemize}
\item Choose a random lattice $L \subseteq \mathbb{R}^n$, and the $i$-th arrival is given by $(\|x_i(L)\|, \pm u_i(L))$.
\item The $i$-th arrival is given by $((p_i/U)^{1/n}, \pm u)$, where $p_i$ is the $i$-th arrival of the Poisson process with intensity $\frac{1}{2}$, and $u$ is a uniform choice of a vector on $S^{n-1}$.
\end{itemize}

Each process imposes a measure on the space of all countable discrete subsets of $\mathcal{R}^n$, say $D(\mathcal{R}^n)$; let us call the measures (by abuse of language) $\mu_n$, and $\wp_n$, respectively. The pertinent $\sigma$-algebra on $D(\mathcal{R}^n)$ is generated by sets of the following form: for $d \in \mathbb{Z}_{>0}$, $0 \leq s_1 < t_1 \leq s_2 < \ldots \leq s_d < t_d$, $a_1, \ldots, a_d \in \mathbb{Z}_{\geq 0}$, $N:= a_1 + \ldots + a_d$, and $\alpha_1, \ldots, \alpha_N \subseteq S^{n-1}/\{\pm 1\}$, define
\begin{align*}
&\mathcal{B}(s_1, t_1, \ldots, s_d, t_d; a_1, \ldots, a_d; \alpha_1, \ldots, \alpha_N) \\
&:=\{S \in D(\mathcal{R}^n): |S \cap ((s_i/U)^{1/n}, (t_i/U)^{1/n}) \times S^{n-1}/\{\pm 1\}| = a_i, \frac{v_j}{\|v_j\|} \in \alpha_j, \forall i,j\},
\end{align*}
where $v_j$ above means the $j$-th vector in $S \cap \bigcup_i ((s_i/U)^{1/n}, (t_i/U)^{1/n}) \times S^{n-1}/\{\pm 1\}$, ordered first by increasing lengths and then by any (fixed) ordering on $S^{n-1}/\{\pm 1\}$; so if all vectors have different lengths, $v_j$ is the $j$-th shortest vector.

Now choose an $o_n(1)$ function $f(n)$. Take any sufficiently large $n$, and choose $d \leq n / \log n$. Choose $0 \leq s_1 < t_1 \leq s_2 < \ldots \leq s_d < t_d$ such that $t_i - s_i = f(n)$ for all $i$, $a_1, \ldots, a_d \in \mathbb{Z}_{\geq 0}$ such that $N = a_1 + \ldots + a_d < Cn$ where $C > 0$ is some absolute constant, and any $\alpha_1, \ldots, \alpha_N \subseteq S^{n-1}/\{\pm 1\}$. Then there exists an absolute constant $c > 0$ such that
\begin{align} \label{eq_star}
& \mu_n(\mathcal{B}(s_1, t_1, \ldots, s_d, t_d; a_1, \ldots, a_d; \alpha_1, \ldots, \alpha_N)) \\
&= \wp_n(\mathcal{B}(s_1, t_1, \ldots, s_d, t_d; a_1, \ldots, a_d; \alpha_1, \ldots, \alpha_N)) + O(e^{-cn}). \notag
\end{align}
\end{Theorem}

\begin{Remark}
It is worth noting that \eqref{eq_star} is meaningful (i.e. the error does not overwhelm the main term) for a wide range of the input variables. Assuming all $\alpha_i = S^{n-1}$, whose effects are easy to see anyway, we have

\begin{equation*}
\wp_n(\ldots) = \prod_{i=1}^d p(f(n),a_i) <\mathrm{exp}\left(d\log(f(n)^a + O(f(n)^{a+1}))\right),
\end{equation*}
where $p(\lambda, m)$ is the probability mass function of the Poisson distribution of mean $\lambda$ evaluated at $m$, and $a := \max a_i$. If $f(n)$ decays and $a$ increases sufficiently slowly, then $\wp_n(\ldots) \gg O(e^{-cn})$. For example, if $d = n/ \log n$, one can take $f(n) = 1/\log n$, and $a = c' \log n / \log\log n$ for some small constant $c' > 0$.
\end{Remark}

Theorem \ref{thm_star} presents a more or less complete picture of $L \cap B(V)$ in high dimensions for $V = d \cdot f(n) = o(n/\log n)$. We suspect that the same would hold for $V < 2n$, \emph{cf.}\ the discussion in Section 3.1 of Venkatesh \cite{V}. Unfortunately, in order for this improvement to happen, the error in \eqref{eq_star} must be strengthened to at least $e^{-O(n\log n)}$, which seems impossible to finesse from the technique that we employ here (see the discussions below, and the remark after Proposition \ref{prop_indep}).

%\begin{quote}
%(*) Choose a sequence $V_n = O(\frac{n}{\log n})$. Then for all sufficiently large $n$, the distribution of $L \cap B(V_n)$ is very close to the distribution of the points $\{0, \pm v_1, \pm v_2, \ldots\}$ in $B(V_n)$ chosen by the following process: choose a random path $(W_1, W_2, \ldots)$ of a Poisson process on $\mathbb{R}_{>0}$ with intensity $\frac{1}{2}$, then pick $v_i$ uniformly from $\partial B(W_i)$.
%\end{quote}
%We suspect that (*) is true for any $V_n < 2n$ as well, but unfortunately it seems impossible to finesse it from the technique we employ here, due to the size of the error terms in the estimates.

Before we go on to present applications of Theorem \ref{thm_star} and other results, it is worth stopping to discuss the obstacles in investigating $L \cap B(V)$ for greater values of $V$ than what we were able to manage in this paper. There exist fundamental technical difficulties in studying $L \cap B(V)$ for $V > 2n$. For instance, the $n$-th and higher moments of $|L \cap B(V)|$ diverge; put another way, we cannot directly observe $n$ or more vectors of $L$ simultaneously. From a thematic point of view, perhaps not unrelated to the technical limitations, $V \approx 2n$ is a transition point between some sort of ``randomness and structure.'' If $V$ equals a ``large'' constant times $n$, say $2.01n$, the vectors in $L \cap B(V)$ must satisfy at least $e^{O(n)}$ linear relations on average. In contrast, if $V < 0.025n$, we prove $L \cap B(V)$ is linearly independent (modulo signs) for most $L$: see Proposition \ref{prop_indep}. Although it is not entirely clear that the plethora of relations imply $L \cap B(V)$ is not random --- because what if the relations have huge coefficients? see also the recent work of Str\"ombergsson and S\"odergren \cite{SS2} --- we see that the algebraic structure inherent in any lattice should play a role in shaping $L \cap B(V)$ as $V$ grows past $2n$. All the results of this paper, as well as the works of Rogers and Schmidt which inspired this paper, is based on the idea that the presence of linear dependencies within $L \cap B(V)$ is negligible when $V$ is small. In order to cross the randomness threshold at $V = (\mbox{some const}) \cdot n$, one needs to start to take seriously the influence of the lattice structure.

Technically, this desired breakthrough would amount to making the error term in \eqref{eq_star} even smaller, with possibly a different $\wp_n$, since by doing so one extends the ranges of the input variables such that the main term \eqref{eq_star} is significantly greater than the error. It would also lead to an improvement on the order of the lower bound on the optimal density of (lattice) sphere packings, or equivalently, on $\sup_{L \in X_n} \nu_1(L)$, which has not happened since Rogers \cite{RogersS} in 1947 despite the time and the variety of the methods taken to attack the problem --- see for instance \cite{Ball}, \cite{D-R}, \cite{KLV}, and \cite{V}. Consider the

\begin{Theorem}[Schmidt \cite{Schmidt1}, also \cite{Schmidt2}] \label{thm_intro_sch}
There exist constants $c, C > 0$ such that, for all $V < Cn$ and $n \geq 13$,
\begin{equation} \label{eq_intro_sch}
\mathrm{Prob}_{\mu_n}(|L\backslash\{0\} \cap B(V)| = 0) = p(V/2,0) + O(e^{-cn}).
\end{equation}
\end{Theorem}

Since $p(V/2,0) = e^{-V/2}$, Theorem \ref{thm_intro_sch} is meaningful when $\frac{V/2}{cn} <1 - \varepsilon$ for any fixed $0 < \varepsilon < 1$, yielding $\sup_{L \in X_n} \nu_1(L) \geq 2cn$ for all sufficiently large $n$, which is up to order of growth in $n$ the best known lower bound to this day. In a similar manner, further shrinking the error term in \eqref{eq_intro_sch} immediately leads to a better lower bound. On the other hand, our proof of Theorem \ref{thm_star} consists of bringing out the full power of the ideas of Schmidt that he used to prove Theorem \ref{thm_intro_sch}. Indeed, not only does \eqref{eq_intro_sch} directly follow from \eqref{eq_star}, but also their error terms come from the same obstacle --- namely the presence of linear dependence in $L \cap B(V)$. Thus the problem of understanding $L \cap B(V)$ for greater values of $V$ is closely related to the sphere packing problem.

\subsection*{A few applications}

Understanding the statistics of $L \cap B(V)$ is closely connected to many problems related to lattices. Below we demonstrate a few applications of our results. Since we wish to spotlight the shape of a random lattice as an intrinsic question, we do not attempt to prove them here --- except for the third one, which seems to garner some interest; see the appendix.

1. Using Theorem \ref{thm_star} we can prove quantitative versions of the various statistics about random Epstein zeta functions in high dimensions studied by S\"odergren (e.g. \cite{SEpstein}, \cite{SS}). In particular, we can estimate their speeds of convergence to the respective limit distributions as $n \rightarrow \infty$. It is unclear at this point what the optimal rate of convergence should be; a recent preprint by Str\"ombergsson and S\"odergren \cite{SS2} suggests a possibility that the ``Poisson-ness'' of $L \cap B(V)$ may hold beyond $V = O(n)$.

%Since the Poisson-ness of $L \cap B(V)$ is expected to hold for not much farther than $V = O(n)$, we expect the rate of convergence obtained from (*) will be very close to the actual one.

2. The Gaussian heuristic $\lambda_i(L) \approx \sqrt{n / 2 \pi e} \cdot i^{1/n}$ for all $1 \leq i \leq n$, where $\lambda_i$ means the $i$-th successive minimum, is an indispensible tool for the cryptanalysis of lattice-based cryptosystems. Using Propositions \ref{prop_s=m} and \ref{prop_lpoisson}, we can compute the distribution function of $\lambda_i$ for all $i < Cn$. Furthermore, Theorem \ref{thm_star} allows us to compute the joint distribution of the first $o(\frac{n}{\log{n}})$ minima.

3. On supporting the Berry-Tabor conjecture on high-dimensional flat tori: here we follow the formulation by Marklof \cite{Marklof}. For $L \in X_n$, define $\nu_i(L) := U \|x_i\|^n$, where $U$ here is the volume of the unit sphere, and $x_i$ is the $i$-th shortest nonzero vector (up to sign) of $L$. Define the \emph{consecutive level spacing distribution} by
\begin{equation*}
P(s,N,L) = \frac{1}{N}\sum_{j=1}^N \delta(s-\nu_{j+1}(L) + \nu_j(L)),
\end{equation*}
where $\delta$ is the Dirac delta as usual. The Berry-Tabor conjecture, in our context, states that
\begin{equation*}
\lim_{N\rightarrow\infty} \int_0^\infty P(s,N,L)h(s)ds = \int_0^\infty \frac{1}{2}e^{-\frac{1}{2}s}h(s)ds
\end{equation*}
for all ``generic'' $L$ and any ``nice enough'' test function $h$. The significance of $\frac{1}{2}e^{-\frac{1}{2}s}$ on the right-hand side is that it is the distribution function of the interarrival times of the Poisson process with intensity $\frac{1}{2}$. Thus the Berry-Tabor conjecture is some kind of a claim that the $\nu_j(L)$'s are locally Poissonian.

Define, for an annulus $A \subseteq \mathbb{R}^n$ of volume $V$ centered at origin,
\begin{equation*}
P(s,A,L) = \frac{1}{V/2}\sum_{j: x_j \in A} \delta(s-\nu_{j+1}(L) + \nu_j(L)).
\end{equation*}

Then we can show using Theorem \ref{thm_star} that, for any annulus $A$ of volume $O(n^\alpha)$ centered at origin, where $\alpha > 0$ is some absolute constant, and any smooth $h: \mathbb{R}_{>0} \rightarrow \mathbb{R}$ with compact support, in particular supported away from $0$,
\begin{equation} \label{eq_BT}
\int_{X_n} \left| \int_0^\infty P(s,A,L)h(s)ds - \int_0^\infty \frac{1}{2}e^{-\frac{1}{2}s}h(s)ds \right| d\mu_n = o_n(1),
\end{equation}
where the $o_n(1)$ term on the right is independent of $A$. By using the Minkowski inequality, we can add up and average \eqref{eq_BT} over arbitrarily many annuli, showing \eqref{eq_BT} for $A$ of any volume.

\eqref{eq_BT} may be interpreted as saying that the Berry-Tabor conjecture on $n$-dimensional flat tori holds up to $o_n(1)$ error on all but $o_n(1)$ percent of all tori.\footnote{Of course, $P(s,A,L)$ and $P(s,N,L)$ are not identical, even if $V = N$. However, a random lattice has $V(A)$ nonzero vectors in $A$ on average, and the standard deviation converges to 0 as $n \rightarrow \infty$. Hence at the cost of another harmless error term we could derive the corresponding statement for $P(s,N,L)$} Indeed, part of the statement of Theorem \ref{thm_star} conveys information about the local spacing statistics of the $\nu_j$'s.

\subsection*{A word about the proofs}

The main tool for studying a random lattice is the Rogers integration formula \cite{Rogersint}, which asserts that
\begin{equation} \label{eq_Rogers}
\int_{X_n} \sum_{x_1, \ldots, x_k \in L - \{0\}} \rho(x_1, \ldots, x_k) d\mu_n = \int_{\mathbb{R}^n} \ldots \int_{\mathbb{R}^n} \rho(x_1, \ldots, x_k) dx_1 \ldots dx_k + \mbox{(error)}
\end{equation}
for $k < n$ and $\rho: (\mathbb{R}^n)^k \rightarrow \mathbb{R}$ a bounded measurable function with compact support. If on the left we sum only over linearly independent $k$-tuples, the error term on the right vanishes.

In Kim \cite{Kim}, the author directly estimated, for a set $S \subseteq \mathbb{R}^n$, the quantity
\begin{equation} \label{eq_prob}
\mathrm{Prob}(|L \cap S| \geq m)
\end{equation}
by sums of expressions of the form \eqref{eq_Rogers} via an inclusion-exclusion argument. The problem is that, the error in \eqref{eq_Rogers} is known to vanish as $n \rightarrow \infty$ only when $k = O(\sqrt{n})$; see Section 9 of Rogers \cite{Rogersmnt} for the estimate. It seems difficult to improve upon Rogers's analysis, which was the reason that in Kim \cite{Kim} we could only look into a set of size $O(\sqrt{n})$.

Schmidt \cite{Schmidt1} presents an ingenious sieve idea which only requires him to estimate a special case of \eqref{eq_Rogers} where one sums over corank $0$ or $\leq 1$ $k$-tuples, in which case the error term can easily be shown to be small for all $k \leq n-1$. Schmidt applies it to estimate \eqref{eq_prob} in case $m = 1$, and obtains Theorem \ref{thm_intro_sch}. In Section 2, we prove a few lemmas that make up the sieve method of Schmidt. And in Section 3, in particular in the proof of Theorem \ref{thm_lmain}, we show that his method can be straightforwardly applied to all $m \leq Cn$, and discuss some of its consequences (Propositions \ref{prop_indep}, \ref{prop_s=m}, and \ref{prop_lpoisson}).

S\"odergren \cite{Sangle} introduces another sieve idea with which he calculates the joint distribution of the lengths and angles of random lattice vectors in a set of constant volume as dimension goes to infinity. It seems possible, by carefully following all his estimates, to verify that his argument in fact works for a set of volume $O(n^{\frac{1}{2} - \varepsilon})$ as well. However, his method also suffers from the difficulty of managing the error terms coming from \eqref{eq_Rogers}. In Section 4, in the proofs of Theorems \ref{thm_amain} and \ref{thm_moment}, we show how we can combine S\"odergren's and Schmidt's sieves together to overcome this issue. Our main result, Theorem \ref{thm_star}, follows as an immediate corollary of Theorem \ref{thm_moment}.

\subsection*{Acknowledgment}
The author would like to thank Junsoo Ha, Seokho Jin, and Akshay Venkatesh for helpful comments and discussions. He also thanks the referee for the careful reading and the numerous corrections of the original manuscript.

\section{Schmidt's sieve}

The following lemma was proved in Proposition 4 of Kim \cite{Kim}.
\begin{Lemma} \label{lemma_binomial}
If $M \geq k$,
\begin{equation*}
\sum_{h=k}^{\alpha} (-1)^{h-k}\binom{h-1}{k-1}\binom{M}{h} \begin{aligned} &\geq 1 \mbox{ if $\alpha - k$ is even,} \\ &\leq 1 \mbox{ if $\alpha - k$ is odd.} \end{aligned}
\end{equation*}
\end{Lemma}

The two lemmas below constitute Schmidt's sieve method. Lemma \ref{lemma_ssieve} is a straightforward generalization of Lemma 2 of \cite{Schmidt1}, and Lemma \ref{lemma_alt} is Lemma 15 of \cite{Schmidt2}.
\begin{Lemma} \label{lemma_ssieve}
Let $0 \leq \sigma_{k-1}, \sigma_k, \sigma_{k+1}, \ldots \leq 1$, and suppose $\sigma_{k+2t-1} \geq \sigma_{k+2t}, \sigma_{k+2t} \leq \sigma_{k+2t+1}$ for all $t \in \mathbb{Z}_{\geq 0}$. Suppose also that $\alpha - k$ is odd. Then
\begin{equation*}
\sum_{h=k}^{\alpha} (-1)^{h-k}\binom{h-1}{k-1}\binom{M}{h}\sigma_h \leq \sigma_{k-1}.
\end{equation*}
Analogously, suppose $0 \leq \tau_{k-1}, \tau_k, \tau_{k+1}, \ldots \leq 1$, with $\tau_{k+2t-1} \leq \tau_{k+2t}, \tau_{k+2t} \geq \tau_{k+2t+1}$ for all $t \in \mathbb{Z}_{\geq 0}$, and this time $\beta - k$ even. Then
\begin{equation*}
\sum_{h=k}^{\beta} (-1)^{h-k}\binom{h-1}{k-1}\binom{M}{h}\tau_h \geq \tau_{k-1}.
\end{equation*}
\end{Lemma}
\begin{proof}
Define
\begin{equation*}
B_h = \begin{cases} (-1)^{h-k}\binom{h-1}{k-1}\binom{M}{h} &\mbox{if $h \geq k$} \\ -1 &\mbox{if $h = k-1$} \end{cases}
\end{equation*}
so that we can write
\begin{equation*}
-\sigma_{k-1} + \sum_{h=k}^{\alpha} (-1)^{h-k}\binom{h-1}{k-1}\binom{M}{h}\sigma_h = \sum_{h=k-1}^{\alpha}\sigma_hB_h.
\end{equation*}

By the summation by parts,
\begin{equation*}
\sum_{h=k-1}^{\alpha}\sigma_hB_h = \sum_{h=k-1}^{\alpha}\left(\sigma_h - \sigma_{h+1}\right)\left(\sum_{j=k-1}^h B_j\right) + \sigma_{\alpha+1}\sum_{j=k-1}^\alpha B_j.
\end{equation*}

Since $\alpha - k$ is odd, we can set $\sigma_{\alpha+1} = 0$. Now observe that, by Lemma \ref{lemma_binomial} and the assumptions on the $\sigma_i$'s, the summand in the first sum on the right-hand side is always non-positive. Thus
\begin{equation*}
\sum_{h=k-1}^{\alpha}\sigma_hB_h \leq 0,
\end{equation*}
which completes the proof of the first inequality.

For the other inequality, start from
\begin{equation*}
\sum_{h=k-1}^{\beta}\tau_hB_h = \sum_{h=k-1}^{\beta}\left(\tau_h - \tau_{h+1}\right)\left(\sum_{j=k-1}^h B_j\right) + \tau_{\beta+1}\sum_{j=k-1}^\beta B_j,
\end{equation*}
and observe that we can set $\tau_{\beta + 1} = 0$ and that the summand in the first sum on the right-hand side is always non-negative.
\end{proof}

\begin{Lemma} \label{lemma_alt}
Let $P$ be a finite set of cardinality $M$. For each $i = 1, \ldots, M$, let $A_i$ be a set consisting of (not necessarily all) subsets of $P$ of cardinality $i$, such that if $E \in A_i$ and $p \in E$, then $E - \{p\} \in A_{i-1}$. Let $A'_i$ be another set consisting of subsets of $P$ of cardinality $i$, such that if $F \in A_{i-1}$ and $p \in P - F$, then $F \cup \{p\} \in A'_{i}$.

Define $\{\sigma_i\}_{i=1}^M$ and $\{\tau_i\}_{i=1}^M$ by the conditions
\begin{align*}
\sigma_i\binom{M}{i} &= \begin{cases} |A_i| &\mbox{if $i$ is odd} \\ |A'_i| &\mbox{if $i$ is even} \end{cases} \\
\tau_i\binom{M}{i} &= \begin{cases} |A_i| &\mbox{if $i$ is even} \\ |A'_i| &\mbox{if $i$ is odd.} \end{cases}
\end{align*}

Then $\sigma_1 \leq \sigma_2 \geq \sigma_3 \leq \ldots$ and $\tau_1 \geq \tau_2 \leq \tau_3 \geq \ldots$.
\end{Lemma}
\begin{proof}
Write $a_i\binom{M}{i} = |A_i|$ and $a'_i\binom{M}{i} = |A'_i|$. It suffices to show that $a_i \leq a'_{i+1}$ and $a'_i \geq a_{i+1}$ for any $i$.

For each $E \in A_i$, there are $M-i$ elements of $A'_{i+1}$ that contains $E$. And for each $F \in A'_{i+1}$, there are at most $i+1$ elements of $A_i$ contained in $F$. Therefore
\begin{equation*}
|A_i| = a_i\binom{M}{i} \leq \frac{i+1}{M-i} \cdot a'_{i+1}\binom{M}{i+1} = a'_{i+1}\binom{M}{i} \Rightarrow a_i \leq a'_{i+1}.
\end{equation*}

Similarly, for each $E \in A_{i+1}$ there are $i+1$ elements in $A'_i$ contained in $E$, and for each $F \in A'_i$ there at most $M-i$ elements in $A_{i+1}$ containing it. So
\begin{equation*}
|A'_i| = a'_i\binom{M}{i} \geq \frac{i+1}{M-i} \cdot a_{i+1}\binom{M}{i+1} = a_{i+1}\binom{M}{i} \Rightarrow a'_i \geq a_{i+1}.
\end{equation*}
\end{proof}

\section{Study of the lengths distribution} \label{sec_length}

Let $L \in X_n$, and let $S \subset \mathbb{R}^n$ be a centrally symmetric ($x \in S \Leftrightarrow -x \in S$) measurable set of volume $V$. Denote by $S'$ the set of all elements of $S$ whose first nonzero coordinate is positive; in particular $0 \not\in S'$. Write $|L \cap S'| = M$, and choose a positive integer $k < n$. Define $p_{S',k}$ so that
\begin{equation*}
p_{S',k}\binom{M}{k} = \left| \{\mbox{unordered, linearly independent $k$-tuples $(x_1, \ldots, x_k)$ in $L \cap S'$}\} \right|.
\end{equation*}

That is, $p_{S',k}$ is the \emph{proportion} of linearly independent $k$-tuples inside $L \cap S'$. Define
\begin{equation*}
P_{S',k}(L) := \begin{cases} p_{S',k} &\mbox{if\ } |L \cap S'| \geq k \\ 0 &\mbox{otherwise,} \end{cases}
\end{equation*}

Similarly, define
\begin{equation*}
P^0_{S',k}(L) := \begin{cases} 1 &\mbox{if $L \cap S'$ contains a linearly independent $k$-tuple} \\ 0 &\mbox{otherwise,} \end{cases}
\end{equation*}
and
\begin{equation*}
P^1_{S',k}(L) := \begin{cases} 1 &\mbox{if $L \cap S'$ contains a corank $\leq 1$ $k$-tuple} \\ 0 &\mbox{otherwise,} \end{cases}
\end{equation*}

The main theorem of this section is
\begin{Theorem} \label{thm_lmain}
Let $1 \leq k < n-1$. Suppose $k$ and $V$ satisfy $V + k \leq \frac{n}{4}\log{\frac{4}{3}}$; for example, $V + k \leq 0.05n$ will do. Then there exists a constant $c > 0$, independent of $n$, $k$ or $V$, such that, for all $n$ sufficiently large,
\begin{align} 
& (1-e^{-n})Q(V/2, k) - e^{-cn}Q(V/2, k-1) \label{eq_lmaineq} \\
& < \int_{X_n} P_{S',k}(L) d\mu_n \leq \int_{X_n} P^0_{S,k}(L) d\mu_n \leq \int_{X_n} P^1_{S',k}(L) d\mu_n \notag \\
& < (1+e^{-n})Q(V/2, k) + e^{-cn}Q(V/2,k-1), \notag
\end{align}
where $Q(N, k)$ denotes the right cumulative distribution function of the Poisson distribution with mean $N$ evaluated at $k$.
\end{Theorem}
\begin{Remark}
 Note that in \eqref{eq_lmaineq}, if $V$ is not too small --- e.g. $V \geq e^{-dn}$ for some $d < c$ is sufficient --- then $Q(V/2, k)$ is the dominant term on both sides.
\end{Remark}
\begin{proof}
The middle inequalities in \eqref{eq_lmaineq} are obvious. We will first prove
\begin{align} 
& (1-e^{-n})Q(V/2, k) - e^{-cn}Q(V/2, k-1) \label{eq_Pbound} \\
&< \int_{X_n} P_{S',k}(L) d\mu_n < (1+e^{-n})Q(V/2, k) + e^{-cn}Q(V/2, k-1), \notag
\end{align}
and argue that the same argument applies to $P^0_{S',k}$ and $P^1_{S',k}$ with only a few tiny modifications.

For $L \in X_n$, define $A_{S',h}(L)$ and $A'_{S',h}(L)$ to be the set of unordered $h$-tuples $(x_1, \ldots, x_h)$ of distinct nonzero elements of $L \cap S'$ with corank 0 and $\leq 1$, respectively. For $0 \leq t < (n-k-1)/2$, define
\begin{equation*}
S_{S',h}(L) = \begin{cases} |A_{S',h}(L)| &\mbox{if $h=k+2t$} \\ |A'_{S',h}(L)| &\mbox{if $h=k+2t+1$,} \end{cases}
\end{equation*}
\begin{equation*}
T_{S',h}(L) = \begin{cases} |A'_{S',h}(L)| &\mbox{if $h=k+2t$} \\ |A_{S',h}(L)| &\mbox{if $h=k+2t+1$.} \end{cases}
\end{equation*}

It is easy to check that $A_{S',h}$ and $A'_{S',h}$ here satisfy the conditions of Lemma \ref{lemma_alt} (by shifting the indices by $k-1$). Therefore, upon defining both $\sigma_{k-1}$ and $\tau_{k-1}$ to equal $p_{S',k}$, Lemma \ref{lemma_ssieve} applies upon writing $S_{S',h}(L) = \sigma_h\binom{M}{h}$ and $T_{S',h}(L) = \tau_h\binom{M}{h}$. As a result, we have
\begin{equation} \label{eq_toestimate}
\sum_{h=k}^{\alpha} (-1)^{h-k}\binom{h-1}{k-1}S_{S',h}(L) \leq P_{S',k}(L) \leq \sum_{h=k}^{\beta} (-1)^{h-k}\binom{h-1}{k-1}T_{S',h}(L)
\end{equation}
for all $\alpha, \beta$ such that $\alpha - k$ is odd and $\beta - k$ is even. It remains to estimate the integral of each side of (\ref{eq_toestimate}). First we will need
\begin{Lemma} \label{lemma_integral}
Suppose $h < n$. Then
\begin{equation*}
\int_{X_n} |A_{S',h}(L)| d\mu_n = \frac{(V/2)^h}{h!},
\end{equation*}
and
\begin{equation*}
\int_{X_n} |A'_{S',h}(L)| - |A_{S',h}(L)| d\mu_n \leq \frac{(V/2)^{h-1}}{(h-1)!}\left( 3^h\left(\frac{3}{4}\right)^\frac{n}{2} + 5^h\left(\frac{1}{2}\right)^n\right).
\end{equation*}
\end{Lemma}
\begin{proof}[Proof of Lemma \ref{lemma_integral}]
These are simple consequences of the Rogers integration formula \cite{Rogersint}. For proofs see Schmidt \cite{Schmidt1}.
\end{proof}

We will only prove the lower bound, since the upper bound is dealt with exactly the same way. By Lemma \ref{lemma_integral}, the left-hand side of \eqref{eq_toestimate} is bounded from below by

\begin{align*}
&\sum_{h=k}^{\alpha} (-1)^{h-k}\binom{h-1}{k-1}\frac{(V/2)^h}{h!} \\
&- \sum_{h=k}^{\alpha} \binom{h-1}{k-1}\frac{(V/2)^{h-1}}{(h-1)!}\left(3^h\left(\frac{3}{4}\right)^{\frac{n}{2}} + 5^h\left(\frac{1}{2}\right)^n\right).
\end{align*}

We may assume $\alpha = n-1$ or $n-2$, whichever one that makes $\alpha - k$ odd. We handle the first sum in the following lemma.

\begin{Lemma} \label{lemma_lmainterm}
Let $n, k, V, \alpha$ as above, so that $k, V \leq C\alpha$ for $C = 0.05$, say. Also write $\lambda = V/2$ for convenience. Then
\begin{equation}\label{dix1}
\left| \sum_{h=k}^{\alpha} (-1)^{h-k}\binom{h-1}{k-1}\frac{\lambda^h}{h!} - Q(\lambda,k) \right| < e^{-n}Q(\lambda,k).
\end{equation}
\end{Lemma}

\begin{proof}[Proof of Lemma \ref{lemma_lmainterm}]
It is a plain fact that
\begin{equation*}
Q(\lambda, k) = e^{-\lambda}\sum_{j=k}^\infty \frac{\lambda^j}{j!}.
\end{equation*}

By Taylor expanding $e^{-\lambda}$ and computing the coefficients of $\lambda$, we find that
\begin{equation*}
Q(\lambda,k) = \sum_{h = k}^\infty (-1)^{h-k}q_h
\end{equation*}
as well, where
\begin{equation*}
q_h := \frac{\lambda^h}{h!}\binom{h-1}{k-1} = \frac{\lambda^h}{h(h-k)!(k-1)!}.
\end{equation*}

It is easy to see that $q_h$ decays at least exponentially for $h \geq \alpha$. Hence the left-hand side of \eqref{dix1} is bounded by
\begin{equation} \label{dix4}
\left| \sum_{h = \alpha+1}^\infty (-1)^{h-k}q_h \right| < q_\alpha < \frac{\lambda^\alpha}{(\alpha-k)!k!}.
\end{equation}

On the other hand, $Q(\lambda, k) > e^{-\lambda} \lambda^k / k!$. Hence the ratio of the error \eqref{dix4} to $Q(\lambda, k)$ is at most
\begin{equation*}
\frac{e^\lambda \lambda^{\alpha - k}}{(\alpha-k)!} \leq \frac{e^{C\alpha}(C\alpha)^{(1-C)\alpha}}{((1-C)\alpha)!} < \frac{e^{\alpha + (1-C)\alpha\log C\alpha}}{e^{(1-C)\alpha\log{(1-C)\alpha}}} = e^{\alpha\left(1 + (1-C)\left(\log\frac{C}{1-C}\right)\right)}.
\end{equation*}
(We used Stirling's approximation in the second inequality above.) It is easily checked on a calculator that $1+(1-C)\left(\log\frac{C}{1-C}\right) < -1.7$ for $C = 0.05$, which is better than what we need.
\end{proof}

As for the second sum --- the error term --- it turns out to be no greater than
\begin{align*}
&\sum_{h=k}^{\alpha} \binom{h-1}{k-1}\frac{(V/2)^{h-1}}{(h-1)!} \cdot 6(3)^{h-1}\left(\frac{3}{4}\right)^\frac{n}{2}\\
&= \frac{6(3/4)^\frac{n}{2}}{(k-1)!}\sum_{h=k}^\alpha \frac{(3V/2)^{h-1}}{(h-k)!}\\
&\leq \frac{6(3/4)^\frac{n}{2}(3V/2)^{k-1}}{(k-1)!}e^{3V/2}\\
&= 6\left(\frac{3}{4}\right)^\frac{n}{2}e^{2V}3^{k-1} \cdot \frac{(V/2)^{k-1}}{(k-1)!}e^{-V/2}\\
&< 6\left(\frac{3}{4}\right)^\frac{n}{2}e^{2V}3^{k-1} \cdot Q(V/2,k-1).
\end{align*}

This tells us that, for $n$ sufficiently large, $V$ and $k$ in the range prescribed by Theorem \ref{thm_lmain}, the error term is exponentially small compared to $Q(V/2,k)$. This completes the proof of \eqref{eq_Pbound}.

Now simply notice that both sides of \eqref{eq_toestimate} are integers. Therefore, if we replace $P_{S',k}$ by $P^0_{S',k}$, \eqref{eq_toestimate} remains true, which proves \eqref{eq_Pbound} for $P^0_{S',k}$. Moreover, the conditions of Lemma \ref{lemma_ssieve} continue to hold if we simply redefine $\sigma_{k-1}$ and $\tau_{k-1}$ to equal $P^1_{S',k}$, which is defined by the equation $|A'_{S',k}| = P^1_{S',k}\binom{M}{k}$, and leave all $S_{S',h}$'s and $T_{S',h}$'s unchanged. So all the argument above goes through, and this shows \eqref{eq_Pbound} for $P^1_{S,k}$. This completes the proof of Theorem \ref{thm_lmain}.
\end{proof}

We next study some of the consequences of Theorem \ref{thm_lmain}. The first is that, roughly speaking, the first $Cn$ shortest vectors of a random lattice is almost always linearly independent.

\begin{Proposition} \label{prop_indep}
Let $S \subset \mathbb{R}^n$ be a centrally symmetric open set of volume $V$. Then there exist constants $c, C > 0$ such that whenever $n$ is sufficiently large and $V \leq Cn$,
\begin{equation} \label{eq_indep}
\mathrm{Prob}(\mbox{$ L \cap S' $ is linearly dependent}) < e^{-c n}.
\end{equation}
\end{Proposition}

\begin{Remark}
This statement explains why Schmidt's sieve works well; the proportion of the dependent tuples are so negligible that Lemma \ref{lemma_ssieve} is really just an inclusion-exclusion principle plus a small error.

On the other hand, if $S = B(V)$, then \eqref{eq_indep} is bounded from \emph{below} by $Ve^{-0.7n}$ for $n$ sufficiently large. A quick sketch of proof: clearly one has
\begin{align*}
\mathrm{Prob}(\mbox{$ L \cap S' $ is linearly dependent}) &\geq \mathrm{Prob}(\mbox{$ L \cap S' $ is linearly dependent and has rank $1$}) \\
&= \int_{X_n} P^1_{S',2}(L) - P^0_{S',2}(L) d\mu_n,
\end{align*}
which equals, in the notation of Schmidt \cite{Schmidt1}, $R^1_2$ for $\rho =$ (the characteristic function of $S'$), which, by Theorem 3 of \cite{Schmidt1}, can easily be shown to be greater than $V/2^{n+1}$.

This indicates the limitation of Schmidt's method: we cannot explore with it a ball of volume greater than $1.4n$, as $Q(V/2, k)$ will typically have size about $e^{-0.7n}$ and thus will be clouded by the error term caused by linear dependence.
\end{Remark}

\begin{proof}
It does no harm to assume $V = Cn$, by enlarging $S$ if necessary, for some constant $C$ to be determined shortly. Also let $k = Cn$, and choose a $C$ so that $V$ and $k$ satisfy the assumptions of Theorem \ref{thm_lmain} e.g. $C = 0.025$ will work.

Clearly, we have
\begin{align*}
& \mathrm{Prob} (\mbox{$L \cap S'$ is linearly dependent}) \\
& \leq \mathrm{Prob}(\mbox{$L \cap S'$ is linearly dependent, and has rank $< k$}) \\
& + \mathrm{Prob}(\mbox{$L \cap S'$ has rank $\geq k$}).
\end{align*}

We will bound the right-hand side term by term. In case $L \cap S'$ is dependent and has rank $= r < k$, $L \cap S'$ has a corank 1 subset of cardinality $r+1$ but has no independent subset of the same cardinality. Therefore
\begin{align*}
& \mathrm{Prob}(\mbox{$L \cap S'$ is linearly dependent, and has rank $< k$}) \\
& \leq \sum_{h=2}^{k} \int_{X_n} \left(P^1_{S',h}(L) - P^0_{S',h}(L)\right) d\mu_n \leq 4ke^{-an},
\end{align*}
for some $a > 0$, by Theorem \ref{thm_lmain}.

Next, write $\lambda := V/2$. Again by Theorem \ref{thm_lmain}, and Stirling's approximation,
\begin{align*}
& \mathrm{Prob}(\mbox{$L \cap S'$ has rank $\geq k$}) \\
& = \mathrm{Prob}(\mbox{$L \cap S'$ contains an independent $k$-tuple}) \\
& < 2Q(\lambda,k) < 4e^{-\lambda}\frac{\lambda^k}{k!} < \frac{e^{-\lambda + k\log\lambda +k}}{e^{k\log k}}.
\end{align*}

Since $\lambda = k/2$, this equals $e^{-k/2 + k\log 1/2 +k} = e^{Cn(-1/2 + \log 1/2 + 1)} < e^{-0.19Cn}$. This completes the proof of the proposition.
\end{proof}

Next, we show that, for the vast majority of $L \in X_n$, the $k$-th shortest nonzero vector of $L$ has length $\lambda_k(L)$ for all $1 \leq k \leq (\mbox{const}) \cdot n$. Here, as usual, $\lambda_k(L) = \inf\{r: \mbox{vectors of $L$ of length at most $r$ span a subspace of dimension at least $k$}\}$ means the $k$-th successive minimum of $L$. From now on, we will sometimes abuse the word ``minimum'' to refer to a lattice vector of that length, which is unique up to signs for $\mu_n$-almost every $L$, as explained in the proof below.

\begin{Proposition} \label{prop_s=m}
Let $C$ and $c$ be as in Proposition \ref{prop_indep}. Then there exists $d > 0$ such that
\begin{equation} \label{eq_s=m}
\mathrm{Prob}(\mbox{$k$-th shortest vector = $k$-th minimum for all $1 \leq k \leq (C/4)n$}) \geq 1- e^{-dn}
\end{equation}
for all sufficiently large $n$.
\end{Proposition}
\begin{proof}
Set $S = B(V)$ where $V = Cn$. We claim that, for some $c' > 0$,
\begin{equation*}
\mathrm{Prob}(|L \cap S'| \geq V/4) \geq 1-e^{-c'n}.
\end{equation*}

Suppose for the moment this is true. By Proposition \ref{prop_indep}, there exists $c > 0$ such that
\begin{equation*}
\mathrm{Prob}(\mbox{$L \cap S'$ is independent with cardinality $\geq V/4$}) \geq 1 - e^{-cn} - e^{-c'n},
\end{equation*}
which proves \eqref{eq_s=m}, because by Lemma 5.1 of S\"odergren \cite{Sangle}, all vectors (modulo $\pm$) of a random lattice have distinct lengths with probability 1.

Hence it remains to prove the claim. By Theorem \ref{thm_lmain}, by adjusting $c$ if necessary,
\begin{align*}
\mathrm{Prob}(|L \cap S'| \leq V/4) &\leq e^{-V/2}\sum_{h \leq \lfloor V/4 \rfloor} \frac{(V/2)^h}{h!} + 2e^{-cn} \\
&< e^{-V/2}\left(\lfloor \frac{V}{4} \rfloor + 1 \right) \cdot \frac{(V/2)^{\lfloor V/4 \rfloor}}{\lfloor V/4 \rfloor!} + 2e^{-cn},
\end{align*}
which, by a similar computation as in the last part of the proof of Lemma \ref{lemma_lmainterm}, can be shown to be less then $e^{-c''n}$ for some $c'' > 0$, as desired.
\end{proof}

%The following proposition is a strengthened version of Theorem 1 of S\"odergren \cite{Sodergren} and Corollary 6.2 of \cite{Sangle}. It amounts to saying that the lengths distribution of the first $o(n)$ --- e.g. $n/\log n$ --- shortest vectors of a random lattice converges in distribution to the first $o(n)$ points of the Poisson process with intensity $1/2$. It would be desirable to replace $o(n)$ by $O(n)$, which is likely the optimal order of growth. Our argument below indicates that one possible approach to this improvement is to strengthen the upper bound of \eqref{eq_indep}, an interesting problem on its own.

Thanks to the results of this section so far, we can now prove the following improvement to the author's previous result \cite{Kim}.

\begin{Proposition} \label{prop_lpoisson}
Let $S \subseteq \mathbb{R}^n$ be a centrally symmetric open set of volume $V$. There exist constants $C, c > 0$ such that, if $n$ is sufficiently large, and if $V, k \leq Cn$, then
\begin{equation*}
Q(V/2, k) - e^{-cn} < \mathrm{Prob}(|L \cap S'| \geq k) < Q(V/2, k) + e^{-cn}.
\end{equation*}

These bounds carry substance if and only if $Q(V/2, k) \gg e^{-cn}$. This is the case, for example, if there exists a constant $D > 0$ such that $e^{-Dn} \gg e^{-cn}$, $V < Dn$, and $V > 2e^{-D/C}k$.
\end{Proposition}
\begin{proof}
Observe that
\begin{align*}
&\mathrm{Prob}(|L \cap S'| \geq k) \\
&= \mathrm{Prob}(\mbox{$|L \cap S'| \geq k$ and independent}) + \mathrm{Prob}(\mbox{$|L \cap S'| \geq k$ and dependent}).
\end{align*}

Theorem \ref{thm_lmain} implies that there exist constants $C, c > 0$ such that the first term equals $Q(V/2,k)$ plus an error of size at most $e^{-cn}$. The second term is bounded between $0$ and $e^{-cn}$ --- adjusting $c$ if necessary --- by Proposition \ref{prop_indep}. This completes the proof of the inequality.

The conditions in the last statement of the proposition are made up simply to ensure that $Q(V/2,k) > e^{-Dn} \gg e^{-cn}$, which can be verified by direct computation.
\end{proof}

At the end of the next section, we will compute the joint distribution of $|L \cap S_i'|$ for many choices of annuli $S_i$ centered at origin, which leads to a quantitative version of S\"odergren's theorem \cite{Sodergren} that the counting process $\{|L \cap B'(V)|, V > 0\}$ converges weakly to a Poisson process as $n \rightarrow \infty$. Furthermore, we also show such distribution is almost independent of the distribution of the mutual angles of the vectors in $L \cap S'$, which corresponds to a quantitative version of another similar theorem of S\"odergren \cite{Sangle}, namely Theorem \ref{thm_star} in the introduction.

\section{Study of the angles and the joint distributions}

We start this section by rewriting
\begin{equation} \label{eq_preGS}
\int_{\mathbb{R}^n} \ldots \int_{\mathbb{R}^n} f(x_1, \ldots, x_N) dx_1 \ldots dx_N
\end{equation}
in terms of a sort of polar coordinates that will be useful to us later.

Choose any $N < n$, and let $x_1, \ldots, x_N \in \mathbb{R}^n$ be arbitrary. For each $i = 1, \ldots, N$, denote by $x^*_i$ the component of $x_i$ orthogonal to $\mathrm{span}(x_1, \ldots, x_{i-1})$, and define $x^{**}_i := x_i - x^*_i$. Then
\begin{align*}
dx^*_i &= \mbox{Lebesgue measure on $\mathrm{span}(x^*_1, \ldots, x^*_{i-1})^\perp$}, \\
dx^{**}_i &= \mbox{Lebesgue measure on $\mathrm{span}(x^*_1, \ldots, x^*_{i-1})$}.
\end{align*}

Clearly $dx_i = dx^*_idx^{**}_i$. We can further rewrite
\begin{align*}
dx^*_i &= \tilde{r}_i^{n-i} d\tilde{r}_i du_i, \\
dx^{**}_i &= \prod_{j=1}^{i-1} d\mu_{i,j},
\end{align*}
where $\tilde{r}_i = \|x^*_i\|$ and $u_i = x^*_i / \|x^*_i\|$ (thus $du_i$ is the Lebesgue measure on the unit sphere $S^{n-i} \subseteq \mathbb{R}^{n-i+1}$), and $\mu_{i,j} = \langle x_i, x^*_j \rangle / \|x^*_j\|$.

With respect to these coordinates, \eqref{eq_preGS} equals
\begin{equation*}
\int \ldots \int f(x_1, \ldots, x_N) \prod_{i=1}^N\left(\tilde{r}^{n-i}_i d\tilde{r}_i du_i \prod_{j=1}^{i-1} d\mu_{i,j}\right).
\end{equation*}

Next, let $r_i : = \|x_i\|$ and $\theta_{i,j}$ be the angle formed by $x_i$ and $x^*_j$. Then
\begin{align*}
\tilde{r}_i^2 &= r^2_i - \sum_{j = 1}^{i-1} \mu^2_{i,j}, \\
\mu_{i,j} &= r_i\sin\theta_{i,1} \ldots \sin\theta_{i,j-1} \cos\theta_{i,j},
\end{align*}
and one computes that \eqref{eq_preGS} equals (\emph{cf}. (4.4) of \cite{Sangle})
\begin{equation} \label{eq_postGS}
\int \ldots \int f(x_1, \ldots, x_N) \prod_{i=1}^N\left(r^{n-1}_i dr_i du_i \prod_{j=1}^{i-1} \sin^{n-j+1}\theta_{i,j}d\theta_{i,j}\right)
\end{equation}
where one integrates the $r$'s over $\mathbb{R}_{\geq 0}$, $u_i$'s over $S^{n-i}$, and $\theta$'s over the interval $[0,\pi]$.

In order to prove our main theorem below, we need to evaluate \eqref{eq_postGS} at functions of a certain form. We will carry this out step by step. Start by supposing that $f: (\mathbb{R}^n)^N \rightarrow \mathbb{R}$ is Borel measurable, bounded, and has a compact support, and that $f$ is of the form
\begin{equation*}
f(x_1, \ldots, x_N) = \rho(r_1, \ldots, r_N) \alpha(\{\theta_{i,j}\}_{1 \leq j < i \leq N})
\end{equation*}
for some $\rho: \mathbb{R}_{\geq 0}^N \rightarrow \mathbb{R}$ and $\alpha: [0,\pi]^{\binom{N}{2}} \rightarrow \mathbb{R}$. Then \eqref{eq_postGS} equals
\begin{align*}
& \prod_{i=1}^{N} S_{n-i} \cdot \int_{r_i \in [0,\infty)} \rho(r_1, \ldots, r_N) \prod_{i=1}^N r^{n-1}_i dr_i \\
& \cdot \int_{\theta_{i,j} \in [0,\pi]} \alpha(\{\theta_{i,j}\}_{1 \leq j < i \leq N}) \prod_{1 \leq j < i \leq N} \sin^{n-j+1} \theta_{i,j} d\theta_{i,j}.
\end{align*}
Here we denoted by $S_{n-i}$ the surface area of $S^{n-i}$. Suppose in addition that
\begin{equation*}
\rho(r_1, \ldots, r_N) = \prod_{i=1}^N I(Ur_i^n \in [s, t])
\end{equation*}
for $s, t \in \mathbb{R}_{\geq 0}$, where $U$ denotes the volume of the unit ball, and $I(\ldots)$ is defined to be equal to $1$ if the condition inside the parenthesis is satisfied and $0$ otherwise. Also write $V = t - s$ for the volume of $B(s,t) := B(t) \backslash B(s)$. Then one can easily find
\begin{align*}
&\int_{\mathbb{R}^n} \ldots \int_{\mathbb{R}^n} f(x_1, \ldots, x_N) dx_1 \ldots dx_N \\
&= V^N \cdot \left( \prod_{i=1}^{N} \frac{S_{n-i}}{S_{n-1}} \int \alpha(\{\theta_{i,j}\}_{1 \leq j < i \leq N}) \prod_{1 \leq j < i \leq N} \sin^{n-j+1} \theta_{i,j} d\theta_{i,j} \right).
\end{align*}

For convenience we will refer to the expression in the parenthesis by $A(\alpha)$, so that
\begin{equation} \label{eq_asep}
\int_{\mathbb{R}^n} \ldots \int_{\mathbb{R}^n} f(x_1, \ldots, x_N) dx_1 \ldots dx_N = A(\alpha)V^N.
\end{equation}

Suppose $\alpha= \alpha_T$ is a characteristic function of a set $T \in [0,\pi]^{\binom{N}{2}}$. We will slightly abuse our language and write $A(\alpha_T) = A(T)$. From our calculations above, it is clear that $A(T)$ equals the probability that $N$ points randomly uniformly chosen from $S^{n-1}$ satisfy the angles relations dictated by $T$.

We still wish to further restrict $f$ so as to account for the fact that all lattices are centrally symmetric. To this end, let us say $T \subseteq [0,\pi]^{\binom{N}{2}}$ is \emph{centrally symmetric} if, for all $i$,
\begin{align*}
& (\ldots, \theta_{i,1}, \ldots, \theta_{i,i-1}, \ldots, \theta_{i+1,i}, \ldots, \theta_{i+2,i}, \ldots) \in T \\ 
&\Leftrightarrow (\ldots, \pi-\theta_{i,1}, \ldots, \pi-\theta_{i,i-1}, \ldots, \pi-\theta_{i+1,i}, \ldots, \pi-\theta_{i+2,i}, \ldots) \in T
\end{align*}
(i.e. change every $\theta$ that has $i$ in the index). For $f$ defined as above, this is equivalent to saying that $f(\ldots, x_i, \ldots) = f(\ldots, -x_i, \ldots)$ for all $i$.

We are now ready to prove the main theorems of this paper. In Theorem \ref{thm_amain}, we estimate the joint distribution of the angles and $|L \cap S'|$ for an annulus $S$ centered at origin. In Theorem \ref{thm_moment}, we compute the joint distribution of the angles and $|L \cap S'_i|$ for many choices of annuli $S_i$. Theorem \ref{thm_star} follows immediately from Theorem \ref{thm_moment}.

\begin{Theorem} \label{thm_amain}
Suppose $f: (\mathbb{R}^n)^N \rightarrow \mathbb{R}$ satisfies all the conditions mentioned earlier in this section: $f$ is a Borel measurable, bounded, and compactly supported function of the form
\begin{equation*}
f(x_1, \ldots, x_N) = \rho(r_1, \ldots, r_N) \alpha_T(\{\theta_{i,j}\}_{1 \leq j < i \leq N})
\end{equation*}
where $T \subseteq [0,\pi]^{\binom{N}{2}}$ is centrally symmetric, and
\begin{equation*}
\rho(r_1, \ldots, r_N) = \prod_{i=1}^N I(Ur_i^n \in [s,t])
\end{equation*}
for some $s,t \in \mathbb{R}_{\geq 0}$. We write $V = t - s$ for the volume of $B(s,t)$.

Let $v_1 = v_1(L)$ be the shortest nonzero vector of $L'=\{v \in L:$ the first nonzero coordinate of $v$ is positive$\}$ that is not contained in $B(s)$, $v_2 = v_2(L)$ the next shortest vector, and so on. Then there exist absolute constants $C, c > 0$ such that for any $V, N \leq Cn$
\begin{equation} \label{eq_amain}
A(T) Q(V/2,N) - e^{-cn} < \int_{X_n} f(v_1, \ldots, v_N) d\mu_n < A(T) Q(V/2,N) + e^{-cn}.
\end{equation}
\end{Theorem}

\begin{proof}
The argument is a hybrid of the sieve ideas of Schmidt \cite{Schmidt1} \cite{Schmidt2} and S\"odergren \cite{Sangle}. It was necessary to incorporate S\"odergren's method because, in order to discuss the angles, we need to give some ordering on the lattice vectors. We will give a proof for the upper bound only, as the lower bound can be proved in the same way, simply by switching the choices of the parities.

Before we start, let us introduce one convenient notation; we will write
\begin{equation*}
f(x_1 < \ldots < x_N) := f(x_1, \ldots, x_N)I(\|x_1\| < \ldots < \|x_N\|).
\end{equation*}

Define, as in the proof of Theorem 5.2 of \cite{Sangle},

\begin{equation} \label{eq_R}
R^n_l(L) := \sum_{(x_1, \ldots, x_{N+l})} f(x_1 < \ldots < x_N)I(\mbox{$(s/U)^{\frac{1}{n}} \leq \|x_{N+1}\| < \ldots < \|x_{N+l}\| < \|x_N\|$}), 
\end{equation}
where we sum over all $(N+l)$-tuples $(x_1, \ldots, x_{N+l})$ of pairwise distinct nonzero vectors of $L'$, such that $\{x_1, \ldots, x_N\}$ is linearly independent, and that $\{x_1, \ldots, x_{N+l}\}$ has corank 0 if $l$ is odd, corank $\leq 1$ if $l$ is even. By Lemma 5.1 of \cite{Sangle}, it does no harm to assume that different vectors of $L'$ have different lengths.

Also define
\begin{equation} \label{eq_S}
S^n_l(L) := \sum_{j=0}^l (-1)^j R^n_j(L).
\end{equation}

The point is that we can rewrite \eqref{eq_S} as
\begin{equation} \label{eq_sssieve}
S^n_l(L) = \sum_{x_1, \ldots, x_N \in L' \atop \mbox{\tiny independent}} f(x_1 < \ldots < x_N) \sum_{j=0}^l (-1)^j \pi_j,
\end{equation}
where $\pi_j$ equals 1 if $j = 0$, otherwise equals the number of the $j$-tuples $(x_{N+1}, \ldots, x_{N+j})$ in $L'- \{x_1, \ldots, x_N\}$ such that $(s/U)^\frac{1}{n} \leq \|x_{N+1}\| < \ldots < \|x_{N+j}\| < \|x_N\|$ and the corank of $\{x_1, \ldots, x_{N+j}\}$ equals 0 if $j$ is odd, $\leq 1$ if $j$ is even.

Fix $x_1, \ldots, x_N \in L'$, and let $P = L \cap B'(s, U\|x_N\|^n) \backslash \{x_1, \ldots, x_N\}$, and write $M = |P|$. For $j = 1, \ldots, l$, let $A_j$ and $A_j'$ to be the set of subsets $\{x_{N+1}, \ldots, x_{N+j}\} \subseteq P$ such that $\{x_1, \ldots, x_{N+j}\}$ has corank 0 and $\leq 1$, respectively. Then Lemma \ref{lemma_alt} applies, and furthermore $\pi_j = |A_j|$ for odd $j$, and $\pi_j = |A'_j|$ for even $j$. Applying Lemma \ref{lemma_ssieve} with $k=1, \alpha=l$ even, $\sigma_0=1$, and $\pi_j = \binom{M}{j}\sigma_j$, we see that the inner summation in \eqref{eq_sssieve} is always nonnegative. Therefore
\begin{equation*}
S^n_l(L) \geq f(v_1,\ldots,v_N).
\end{equation*}
(To obtain the other inequality $S^n_l(L) \leq f(v_1,\ldots,v_N)$ for an odd $l$, switch all the parities in the above argument, and apply Lemmas \ref{lemma_alt} and \ref{lemma_ssieve} in the same manner, setting $k = 1$, $\alpha = l$, $\pi_0 = 1$ as before, but using the statements in $\tau_j$'s instead of those in $\sigma_j$'s. Then one has $\sum_{j=0}^l (-1)^j\pi_j \leq 0$, unless $(x_1, \ldots, x_N) = (v_1, \ldots, v_N)$, in which case $\pi_0=1$ but $\pi_1 = \ldots = \pi_l = 0$; in particular, Lemma \ref{lemma_ssieve} does not apply in this case.)

It remains to estimate the average of $S^n_l$. We first estimate $\int_{X_n} R^n_j d\mu_n$. The contribution to $\int_{X_n} R^n_j d\mu_n$ of the corank 0 tuples equals
\begin{align} \label{eq_comb}
&\frac{1}{2^{N+j}}\int \ldots \int f(x_{1} < \ldots < x_{N})I(\mbox{$\left(\frac{s}{U}\right)^\frac{1}{n} \leq \|x_{N+1}\| < \ldots < \|x_{N+j}\| < \|x_{N}\|$}) dx_1 \ldots dx_{N+j} \\ \notag
&= \frac{1}{2^{N+j}}\frac{(N+j-1)!}{(N-1)!j!(N+j)!} \int \ldots \int f(x_{1}, \ldots, x_{N})I(\mbox{$\left(\frac{s}{U}\right)^\frac{1}{n} \leq \|x_{N+1}\|, \ldots, \|x_{N+j}\| \leq \left(\frac{t}{U}\right)^\frac{1}{n}$}) \\ \notag
&dx_1 \ldots dx_{N+j}\\ \notag
&= A(T) \cdot \binom{N+j-1}{N-1} \frac{\left(V/2\right)^{N+j}}{(N+j)!}.
\end{align}

The first line is the Rogers integration formula \eqref{eq_Rogers}; the $\frac{1}{2^{N+j}}$ factor is to account for the sign variations. The factorials in the second line appear because, of all $(N+j)!$ orderings of the $(N+j)$-tuple $(x_1, \ldots, x_{N+j})$, exactly $\frac{(N+j-1)!}{(N-1)!j!}$ of them satisfy the conditions $\|x_1\| < \ldots < \|x_N\|$ and $\|x_{N+1}\| < \ldots < \|x_{N+j}\| < \|x_N\|$. The third line follows by \eqref{eq_asep}.

As for the contribution of corank 1 tuples, we simply bound it from above by
\begin{equation*}
\frac{(V/2)^{N+j-1}}{(N+j-1)!}\left( 3^{N+j}\left(\frac{3}{4}\right)^\frac{n}{2} + 5^{N+j}\left(\frac{1}{2}\right)^n\right)
\end{equation*}
using Lemma \ref{lemma_integral}.

Now take $l$ to be the greatest even number such that $N+l < n$. So far, we have shown that $\int_{X_n} S^n_l(L) d\mu_n$ is bounded from above by
\begin{align*}
& A(T) \cdot \sum_{j=0}^{l} (-1)^{j}\binom{N+j-1}{N-1}\frac{(V/2)^{N+j}}{(N+j)!} \\
& + \sum_{j=0}^{l} \frac{(V/2)^{N+j-1}}{(N+j-1)!}\left(3^{N+j}\left(\frac{3}{4}\right)^{\frac{n}{2}} + 5^{N+j}\left(\frac{1}{2}\right)^n\right).
\end{align*}

But estimating this quantity is already done in the proof of Theorem \ref{thm_lmain}. Furthermore, thanks to Proposition \ref{prop_indep}, at the cost of an exponentially small error in $n$, we could ignore the cases where the $v_i$'s are linearly dependent. This proves the upper bound of \eqref{eq_amain}. As mentioned earlier in the proof, the lower bound can be proved similarly.
\end{proof}

\begin{Remark}
For the vectors $x_1, \ldots, x_N$ randomly and uniformly chosen from $S^{n-1}$, one may ask about the distribution of $\phi_{i,j} := \arccos (\frac{x_i}{\|x_i\|} \cdot \frac{x_j}{\|x_j\|})$ instead of $\theta_{i,j}$, and whether the distribution converges to a product of error functions, as in S\"odergren \cite{Sangle}. In Section 3 of \cite{Sangle}, he first computes the distribution of $\{\theta_{i,j}\}_{1 \leq j < i \leq N}$, and then shows that $\phi_{i,j}$ may be approximated by $\theta_{i,j}$. Near the end of the section, he shows $\tilde{\alpha}_{i,j} = t_{i,j} + O(n^{-1/2})$ for each $1 \leq i < j \leq N$, where $\tilde{\alpha}_{i,j}$ and $t_{i,j}$ are some normalizations of $\phi_{i,j}$ and $\theta_{i,j}$ respectively. This approximation continues to be valid uniformly for all $1 \leq i < j \leq N$ if $N = o(\sqrt{n})$ --- i.e. all the error terms $O(n^{-1/2})$ can be controlled simultaneously --- but for a higher growth order this is no longer true.

Regarding the error function, S\"odergren \cite{Sangle} shows that, for fixed $\{t_{i,j}\}_{1 \leq j < i \leq N}$,
\begin{equation*}
\prod_{1 \leq j < i \leq N} \cos^{n-i-1}\left(\frac{t_{i,j}}{\sqrt{n}}\right) \sim \prod_{1 \leq j < i \leq N} e^{-t_{i,j}^2/2}
\end{equation*}
as $n \rightarrow \infty$, by interpreting $\cos\left({t_{i,j}}/{\sqrt{n}}\right) = \left(1 - t_{i,j}^2/2n + O(n^{-2})\right)$. Again, this holds for $N = o(\sqrt{n})$, but not for a higher growth order, because then the product of the error terms becomes a non-vanishing quantity in the $n$ limit.
\end{Remark}

\begin{Theorem} \label{thm_moment}
Let $d > 0$ and $k_1, \ldots, k_d \geq 0$ be integers. Write $N = \sum k_i$. Choose $d$ disjoint annuli $S_i = B(s_i, t_i)$ in the increasing order (so that $t_{i} \leq s_{i+1}$), and denote by $V_i = t_i - s_i$ the volume of $S_i$. Suppose $f: (\mathbb{R}^n)^N \rightarrow \mathbb{R}$ is of the form
\begin{equation*}
f(x_1, \ldots, x_N) = \rho(r_1, \ldots, r_N)\alpha_T(\{\theta_{i,j}\}_{1 \leq j < i \leq N}),
\end{equation*}
where $T \subseteq [0,\pi]^{\binom{N}{2}}$ is centrally symmetric, and $\rho$ is of the form
\begin{equation*}
\rho(r_1, \ldots, r_N) = I(\mbox{$r_{k_1+\ldots+k_i+1}, \ldots, r_{k_1+\ldots+k_i+k_{i+1}} \in [s_i, t_i]$ for all $1 \leq i \leq d$}).
\end{equation*}

Now choose any constant $D>0$. Then there exist constants $C, c > 0$, depending only on $D$, such that, provided $d < \frac{Dn}{\log n}$ and $N, \sum_{i=1}^d{V_i/2} < Cn$,
\begin{equation} \label{eq_moment}
 \int_{X_n} f(v_1, \ldots, v_N) d\mu_n = A(T)p(\frac{V_1}{2}, k_1) \ldots p(\frac{V_{d-1}}{2}, k_{d-1})Q(\frac{V_d}{2}, k_d) + O(e^{-cn}),
\end{equation}
where $v_i = v_i(L)$ be the $i$-th shortest vector of $L$ in $\bigcup S'_i$ as in the statement of Theorem \ref{thm_amain}, and $p(\lambda, k)$ is the probability mass function of the Poisson distribution of mean $\lambda$.
\end{Theorem}
%where $P_l := \{(\alpha_1, \ldots, \alpha_d) \in \mathbb{Z}^d_{\geq 0} : \sum \alpha_i \leq l\}$, $l_1$ is odd and $l_2$ is even satisfying $N + l_i < n$, and

%Compare \eqref{eq_moment} with
%\begin{align*}
%&p(\frac{V_1}{2}, k_1) \ldots p(\frac{V_{d-1}}{2}, k_{d-1})Q(\frac{V_d}{2}, k_d) \\
%&= \sum_{(\alpha_1, \ldots, \alpha_d) \in \mathbb{Z}^d_{\geq 0}} (-1)^{\sum \alpha_i} C(\alpha_1, \ldots, \alpha_d) \prod \left(\frac{V_i}{2}\right)^{k_i + \alpha_i},
%\end{align*}
%where $p(\lambda, k)$ is the probability mass function of the Poisson distribution of mean $\lambda$.

\begin{proof}
All the arguments in the proof of Theorem \ref{thm_amain} go through with $s_1 = s$ and $t_d = t$, except for the part corresponding to \eqref{eq_comb} where we check that $\int_{X_n}R^n_j d\mu_n$ has the correct coefficient that matches that of the Taylor series of $Q(V/2, k)$. The $A(T)$ factor, however, is still easily seen to be separated from the rest, so throughout this proof we assume $T = [0,\pi]^{\binom{N}{2}}$ and ignore the angles issue altogether.

We need to compare the coefficients of $(V_1/2)^{k_1 + \alpha_1} \ldots (V_d/2)^{k_d + \alpha_d}$ on each side of \eqref{eq_moment}. The main term on the right-hand side equals
\begin{align} \label{eq_pexp}
& \left(\sum_{h_1 = 0}^\infty (-1)^{h_1}\frac{(V_1/2)^{k_1+h_1}}{k_1!h_1!}\right) \ldots \left(\sum_{h_{d-1} = 0}^\infty (-1)^{h_{d-1}}\frac{(V_{d-1}/2)^{k_{d-1}+h_{d-1}}}{k_{d-1}!h_{d-1}!}\right) \\
&\times \left(\sum_{h_d = k_d}^\infty (-1)^{h_d-k_d} \binom{h_d-1}{k_d-1} \frac{(V_d/2)^{h_d}}{h_d!}\right). \notag
\end{align}

It can be checked easily that the $(k_1 + \alpha_1, \ldots, k_d + \alpha_d)$-th coefficient of this expression equals $(-1)^{\sum \alpha_i}$ times
\begin{equation*}
C(\alpha_1, \ldots, \alpha_d) = \frac{\binom{k_1+\alpha_1}{k_1} \ldots \binom{k_{d-1}+\alpha_{d-1}}{k_{d-1}}\binom{k_d+\alpha_d-1}{k_d -1}}{(k_1+\alpha_1)! \ldots (k_d + \alpha_d)!}.
\end{equation*}

On the left-hand side of \eqref{eq_moment}, the term with degree $(k_1 + \alpha_1, \ldots, k_d + \alpha_d)$ comes from $\int R^n_j$ for $j = \alpha_1 + \ldots + \alpha_d$. For each $(N + j)$-tuple $(x_1, \ldots, x_{N+j})$ which has $k_i + \alpha_i$ elements in $B(s_i, t_i)$ --- so that it contributes to the term with the said degree upon integrating --- there are $(k_1+\alpha_1)! \ldots (k_d + \alpha_d)!$ reorderings of the tuple that make $f(x_1, \ldots, x_N)I((s/U)^\frac{1}{n} \leq x_{N+1}, \ldots, x_{N+j})$ nonzero, and $\binom{k_1+\alpha_1}{k_1} \ldots \binom{k_{d-1}+\alpha_{d-1}}{k_{d-1}}\binom{k_d+\alpha_d-1}{k_d -1}$ reorderings that satisfy $x_1 < \ldots < x_N$ and $x_{N+1} < \ldots < x_{N+j} < x_N$. This shows that the coefficient from the left-hand side of \eqref{eq_moment} agrees with $C(\alpha_1, \ldots, \alpha_d)$ above. Hence the intended main term of $\int f(v_1, \ldots, v_N)$ equals
\begin{equation} \label{eq_pcomp}
\sum_{(\alpha_i) \in P} (-1)^{\sum \alpha_i} C(\alpha_1, \ldots, \alpha_d) \prod_{i=1}^d \left(\frac{V_i}{2}\right)^{k_i + \alpha_i},
\end{equation}
where $P := \{(\alpha_1, \ldots, \alpha_d) \in \mathbb{Z}^d_{\geq 0} : \sum \alpha_i \leq l\}$ with $l \approx (1-C)n$ --- this $l$ plays the same role as in the proof of Theorem \ref{thm_amain}.

It remains to show that \eqref{eq_pcomp} is close to \eqref{eq_pexp} given all the assumptions. To this end, we consider the terms of total degree $N+j$ for each $j$. We will show that for $j > l$, the sum of the coefficients of such terms are at most $e^{-Lj}$ in absolute value for some constant $L > 0$. This will prove \eqref{eq_moment}.

We start by noticing that
\begin{equation} \label{eq_term}
C(\alpha_1, \ldots, \alpha_d)\prod_{i=1}^d \left(\frac{V_i}{2}\right)^{k_i + \alpha_i} = \frac{k_d}{k_d + \alpha_d} \left(\prod_{i=1}^d \frac{(V_i/2)^{k_i}}{k_i!}\right) \cdot \left(\prod_{i=1}^d \frac{(V_i/2)^{\alpha_i}}{\alpha_i!}\right).
\end{equation}

Suppose $\sum \alpha_i = j$. Using the Lagrange multiplier, we see that the expression of form $\prod_{i=1}^d {x_i}^{a_i}$, with $a_i > 0$ fixed and $\sum x_i$ constant, is maximized when $x_i = \frac{\sum x_j}{\sum a_j} \cdot a_i$. By applying this fact to the first product on the right-hand side of \eqref{eq_term}, and then using the Stirling's formula, we find that it is at most $\mathrm{exp}\left(N(1+\log Cn - \log N)\right)$; since this is an increasing function of $N$ for $N < Cn$, the supremum is attained for $N = Cn$. Similarly, the second product is at most $\mathrm{exp}\left(j (1+\log (Cn/j))\right)$. (If any $\alpha_i$ or $k_i$ equals $0$ it only improves the estimate.) Also, since there are up to correct order $j^{d-1}$ terms of total degree $N+j$, the sum of \eqref{eq_term} is more or less bounded by
\begin{equation} \label{ribbit1}
\exp(Cn + (d-1)\log j + j(1+\log(Cn/j))).
\end{equation}

As promised just earlier, we claim that this is at most $e^{-Lj}$ for some $L > 0$. Recall $j > l \approx (1-C)n$. Evaluated at $j = (1-C)n$, the expression inside the exponent is at most
\begin{equation*}
(D+1)n + (1-C)n\log\frac{C}{1-C},
\end{equation*}
which is less than a negative constant times $j$, say $-Lj$, if $C$ is sufficiently small. In addition, differentiating the exponent of \eqref{ribbit1} by $j$, we obtain that the derivative is at most
\begin{equation*}
\frac{D}{(1-C)\log n} +\log{Cn} - \log j < \frac{D}{(1-C)\log n} + \log\frac{C}{1-C},
\end{equation*}
which is less than a negative constant, again say $-L$, by adjusting the constants if necessary. Therefore \eqref{ribbit1} is less than $-e^{Lj}$ for all $ j > l$, as desired. This completes the proof of the theorem.
\end{proof}

Theorem \ref{thm_star} is now a quick corollary of Theorem \ref{thm_moment}. \eqref{eq_star} is nothing more than a paraphrase of \eqref{eq_moment}, except that the former prescribes the directions of every single vector via the $\alpha_i$'s, whereas the latter specifies only the relative angles of the vectors;  but this does not make \eqref{eq_moment} a weaker statement, because $\mu_n$ is invariant under rotation.

\appendix
\section{Proof of (1.3)}

We recall (1.3) from the main text:

\begin{equation} \tag{1.3}
\int_{X_n} \left| \int_0^\infty P(s,A,L)h(s)ds - \int_0^\infty \frac{1}{2}e^{-\frac{1}{2}s}h(s)ds \right| d\mu_n = o_n(1),
\end{equation}
where
\begin{equation*}
P(s,A,L) = \frac{1}{V/2}\sum_{j: x_j \in A} \delta(s-\nu_{j+1}(L) + \nu_j(L)),
\end{equation*}
(recall $x_j = x_j(L)$ is the $j$-th shortest nonzero vector of $L$, and $\nu_j = U_n\|x_j\|^n$, where $U_n = \frac{\pi^{n/2}}{\Gamma(n/2+1)}$) and $h: \mathbb{R}_{>0} \rightarrow \mathbb{R}$ is a compactly supported smooth function, in particular away from zero. 

The goal of this appendix is to prove this statement. For $\varepsilon > 0$, define $h_n(s) = h(k\varepsilon)$ if $s \in [k\varepsilon, (k+1)\varepsilon)$ (we understand $h(0) = 0$). Later we will determine $\varepsilon$ to be an explicit $o_n(1)$ function, so it suffices to prove (1.3) with $h$ replaced by $h_n$.

We claim also that we can replace the definition of $P(s,A,L)$ by
\begin{equation} \label{eq_PL}
P(s,N(A),L) = \frac{1}{N(A)}\sum_{j: x_j \in A} \delta(s-\nu_{j+1}(L) + \nu_j(L)),
\end{equation}
where $N(A) = \frac{1}{2}|L \backslash \{0\} \cap A|$. By Proposition 3.3, if $1 < V < Cn$, the $\mu_n$-average of $N(A)$ equals $V/2$, and the standard deviation is $O(\sqrt{V})$. Hence by Chebyshev's inequality, $|N(A) - V/2| < V^{2/3}$, thus $|N(A)^{-1} - (V/2)^{-1}| = O(V^{-4/3})$, except on a set of measure $O(V^{-1/3})$. On that exceptional set, the integral (1.3) vanishes as $n \rightarrow \infty$ provided $V \rightarrow \infty$, because the integrand is bounded (because $h$ is supported away from zero).
% note: actually, it's bounded by $\|h(x)/x\|_\infty + \int\frac{1}{2}e^{-\frac{1}{2}s}h(s)ds$.

Therefore our goal is to prove
\begin{equation} \label{eq_goal}
\int_{X_n} \left| \int_0^\infty P(s,N(A),L)h_n(s)ds - \int_0^\infty \frac{1}{2}e^{-\frac{1}{2}s}h_n(s)ds \right| d\mu_n = o_n(1).
\end{equation}

\vspace{8mm}

We will first show, by Theorem 1.4, that $P(s, N(A), L)$ here can be replaced by
\begin{equation} \label{eq_PP}
P(s, N(A), w) = \frac{1}{N(A)}\sum_{j: x_j \in A} \delta(s-\nu_{j+1}(w) + \nu_j(w)),
\end{equation}
where $w \in D(\mathcal{R}^n)$, and now $N(A) = |w \cap A/\{\pm 1\}|$, $x_j$ is the $j$-th shortest nonzero vector of $w$, and $\nu_j := U_n\|x_j\|^n$ (if one interprets $L \in X_n$ as an element of $D(\mathcal{R}^n)$ in the natural way, then the definitions \eqref{eq_PL} and \eqref{eq_PP} coincide), and that the integration over $(X_n, \mu_n)$ can be replaced by the integration over $(D(\mathcal{R}^n), \wp_n)$. Once this is done, proving \eqref{eq_goal} reduces to an exercise in probability theory.

Suppose $d \leq n/\log n$ is some positive integer, and $\lambda := \lambda(n)$ is an $o_n(1)$ function such that $\varepsilon$ is an integer multiple of $\lambda$ and $m := \varepsilon/\lambda \rightarrow \infty$ as $n \rightarrow \infty$, both of which we will explicitly determine later. Divide the annulus $A$ into $d$ consecutive annuli $B(s_1,t_1), \ldots, B(s_d, t_d)$ of equal volume $2\lambda$. For any $d$ nonnegative integers $a_1, \ldots, a_d$, let us write for shorthand
\begin{equation*}
\mathcal{B}(a_1, \ldots, a_d) := \mathcal{B}(s_1, t_1, \ldots, a_1, \ldots, a_d, S^{n-1}, \ldots, S^{n-1}).
\end{equation*}

Rewrite the left-hand side of \eqref{eq_goal} as
\begin{equation} \label{eq_goal_re}
\sum_{a_1, \ldots, a_d \in \mathbb{Z}_{\geq 0}} \int_{X_n \cap \mathcal{B}(a_1, \ldots, a_d)} \left| \sum_{k = 0}^\infty  \int_{k\varepsilon}^{(k+1)\varepsilon} h_n(s)\left(P(s,N(A),L) - \frac{1}{2}e^{-\frac{1}{2}s}\right)ds \right| d\mu_n.
\end{equation}

We truncate some of the outermost summation. By Theorem 1.4, the sum of the volumes of $\mathcal{B}(a_1, \ldots, a_d)$ for which $0 \leq a_1, \ldots, a_d \leq 1$ equals
\begin{equation*}
\left(p(\lambda,0) + p(\lambda,1)\right)^d + O(2^de^{-cn}) = \left(1 - \frac{\lambda^2}{2} + O(\lambda^{3}) \right)^d +O(2^de^{-cn}),
\end{equation*}
which approaches $1$ if $d\log 2 -cn < 0$ and $d\lambda^2 \rightarrow 0$, as $n \rightarrow \infty$. Therefore, supposing these, we can sum \eqref{eq_goal_re} over just $a_1, \ldots, a_d = \mbox{$0$ or $1$}$.

Next, given $a_1, \ldots, a_d = \mbox{$0$ or $1$}$, let $n_1$ be the smallest index such that $a_{n_1} = 1$, $n_2$ be the next such smallest index, and so on. We would like to replace $P(s, N(A), L)$ in the innermost parenthesis of \eqref{eq_goal_re} with its ``discretized'' version
\begin{equation*}
P(k, a_1, \ldots, a_d) = \frac{1}{\sum a_i} (\mbox{$\#$ of $i$'s such that $n_{i+1} - n_i \in [km, (k+1)m)$}),
\end{equation*}
(recall $m = \varepsilon/\lambda$ is an integer) and show that
\begin{equation*}
\int_{k\varepsilon}^{(k+1)\varepsilon} P(s, N(A), L) ds = P(k, a_1, \ldots, a_d)
\end{equation*}
for all $k$. This holds, except when $(a_1, \ldots, a_d)$ lies in ``boundary cases,'' i.e. $n_{i+1} - n_i = km$ for some positive integers $i$ and $k$, so that $\nu_{i+1} - \nu_i$ could be contained in $[(k-1)\varepsilon, k\varepsilon)$ rather than in $[k\varepsilon, (k+1)\varepsilon)$, making difference on the left-hand side of the above expression.

We resolve this situation by correcting the inner integral of \eqref{eq_goal_re} by some function $g_n(s,L)$, and showing that its integral is small. For each $L \in X_n$, $g_n(s,L)$ is defined as follows: if there are $t_k$ indices $i$ such that $n_{i+1} - n_i = km$ but $\nu_{i+1} - \nu_i \in [(k-1)\varepsilon, k\varepsilon)$, set $g_n(s,L) = t_k(h_n(k\varepsilon) - h_n((k-1)\varepsilon))$ for $s \in [k\varepsilon, (k+1)\varepsilon)$. Note $h_n(k\varepsilon) - h_n((k-1)\varepsilon) = O(\varepsilon)$ because $h$ is smooth; $\int_{0}^{\infty} \left| g_n(s,L) \right| ds = O(t\varepsilon^2)$, where $t = \sum_k t_k$ is the number of total corrections needed.

We need to show that
\begin{equation*}
\sum_{a_1, \ldots, a_d = \mathrm{0\,or\,1}} \int_{X_n \cap \mathcal{B}(a_1, \ldots, a_d)} \int_{0}^{\infty} \left| g_n(s,L) \right| ds d\mu_n
\end{equation*}
vanishes as $n \rightarrow \infty$. By Theorem 1.4, this is bounded by
\begin{equation*}
\sum_{a_1, \ldots, a_d = \mathrm{0\,or\,1}} p(\lambda,0)^{d-f}p(\lambda,1)^f\left|\{i : m | (n_{i+1} - n_i)\}\right| \cdot O(\varepsilon^2) + O(d2^de^{-cn}),
\end{equation*}
where $f = f(a_1, \ldots, a_d) = |\{i:a_i = 1\}|$. Rewriting the sum in terms of $f$, and leaving out the irrelevant parts, it turns out that we need to show
\begin{equation}\label{eq_hardpart0}
\sum_{f=1}^d p(\lambda,0)^{d-f}p(\lambda,1)^f \varepsilon^2 \sum_{1 \leq n_1 < \ldots < n_f \leq d} \left|\{i: m | (n_{i+1}-n_i) \}\right|
\end{equation}
is small. The idea is to compare the inner summation with $\binom{d}{f}$, and then appeal to the binomial theorem. We start by rewriting the inner sum as
\begin{equation}\label{eq_hardpart1}
\sum_{1 \leq n_1 < \ldots < n_f \leq d} \sum_{j=1}^{f-1} I(m|n_{j+1}-n_j) = \sum_{j=1}^{f-1}\sum_{1 \leq n_1 < \ldots < n_f \leq d} I(m|n_{j+1}-n_j).
\end{equation}

Fix a $1 \leq j \leq f-1$, and consider the inner sum on the right-hand side. If $n_{j+1} - n_j = km$ for some $k$, this determines $n_{j+1} = n_j + km$, and the choices for the remaining $n_i$'s, $i \neq j+1$, if possible at all, must be made within a set of size $d - km$. Therefore, \eqref{eq_hardpart1} equals
\begin{align*}
&\sum_{j=1}^{f-1}\sum_{k=1}^{\lfloor d/m \rfloor} \binom{d-km}{f-1} = \frac{1}{m}\sum_{j=1}^{f-1}\sum_{k=1}^{\lfloor d/m \rfloor} m\binom{d-km}{f-1} \\
&\leq \frac{1}{m}\sum_{j=1}^{f-1}\sum_{l=1}^{d} \binom{d-l}{f-1} \leq \frac{1}{m}f\binom{d}{f},
\end{align*}
by the hockey-stick identity. Note the saving by the factor of $1/m$ compared to the trivial bound. Applying this estimate, and the (derivative) of the binomial theorem
\begin{equation*}
d(\alpha + \beta)^{d-1} = \sum_{f=1}^d \alpha^{d-f}\beta^{f-1}f\binom{d}{f}
\end{equation*}
to \eqref{eq_hardpart0}, we see that it is bounded by
\begin{align*}
& d(p(\lambda,0) + p(\lambda,1))^{d-1} \cdot p(\lambda,1)\varepsilon^2/m \\
&< d(1-\lambda^2/2 + O(\lambda^3))^{(d-1)\lambda^2/\lambda^2} \cdot \lambda^2\varepsilon \\
&\ll d\lambda^2\varepsilon
\end{align*}
as desired, by our assumption that $d\lambda^2 \rightarrow 0$ as $n \rightarrow \infty$.

Hence we conclude that \eqref{eq_goal_re} equals, again by Theorem 1.4,

%Rewriting this sum in terms of $f$ and $y_i := n_{i+1} - n_i$ and estimating, we have

%\begin{align*}
%&\leq \sum_{f=1}^d \left((1-\lambda)^{d-f} \lambda^f + O(\lambda^{f+1})\right)  \sum_{y_1, \ldots, y_{f-1} \leq d} |\{i: m | y_i\}| \cdot O(\varepsilon) \\
%&\leq \sum_{f=1}^d O(fd\lambda^{f+1}) = O(d\lambda^2),
%&\leq \sum_{f=1}^d \left((1-\lambda)^{d-f} \lambda^f + O(\lambda^{f+1})\right) \cdot fd\lambda \cdot O(1).
%&\leq \sum_{f=0}^d (1-\lambda)^{d-f} \lambda^f f \cdot \frac{d}{m} \cdot O(\varepsilon)\  \leq\  \sum_{f=0}^d (1-\lambda)^{d-f} \lambda^f  f \cdot d\lambda \cdot O(1).
%\end{align*}
%which goes to $0$ by our prior assumption that $d\lambda^2 \rightarrow 0$.

%But the error from the boundary cases can be covered by correcting the inner integral by some function $g_n(k, L)$ (which also depends on our parsing of $A$) with compact support in the first variable and $g_n(k, L) = O(\varepsilon)$, because, as $h$ is smooth, $h_n(k\varepsilon)$ and $h_n((k+1)\varepsilon)$ differs by at most $B\varepsilon$ for some $B > 0$. Hence \eqref{eq_goal_re} equals

\begin{align*} 
&\sum_{a_1, \ldots, a_d = \mathrm{0\,or\,1}} \int_{X_n \cap \mathcal{B}(a_1, \ldots, a_d)} \left| \sum_{k = 0}^\infty  \int_{k\varepsilon}^{(k+1)\varepsilon}  h_n(k\varepsilon)\left(P(s,N(A),L) - \frac{1}{2}e^{-\frac{1}{2}s}\right) - g_n(k\varepsilon, L) ds \right| d\mu_n \\
&= \sum_{a_1, \ldots, a_d = \mathrm{0\,or\,1}} \left(\prod_{i=1}^{d} p(\lambda, a_i) + O(e^{-cn})\right) \left| \sum_{k = 0}^\infty   h_n(k\varepsilon)P(k, a_1, \ldots, a_d) - \int_{0}^{\infty}\frac{1}{2}e^{-\frac{1}{2}s}h_n(k\varepsilon)ds\right| \notag \\
& + O(d2^de^{-cn}) + O(d\lambda^2\varepsilon) \notag \\
&= \sum_{a_1, \ldots, a_d = \mathrm{0\,or\,1}} \prod_{i=1}^{d} p(\lambda, a_i) \left| \sum_{k = 0}^\infty  h_n(k\varepsilon)P(k, a_1, \ldots, a_d) - \int_{0}^{\infty}\frac{1}{2}e^{-\frac{1}{2}s}h_n(k\varepsilon)ds \right| \\
& + O(d2^de^{-cn}) + O(d\lambda^2\varepsilon). \notag
\end{align*}

By reverse-engineering what we have done to get this far from the left-hand side of \eqref{eq_goal}, this equals, up to errors that vanish in the limit, 
\begin{equation} \label{eq_goal_rere}
\int_{D(\mathcal{R}^n)} \left| \int_0^\infty P(s,N(A),w)h_n(s)ds - \int_0^\infty \frac{1}{2}e^{-\frac{1}{2}s}h_n(s)ds \right| d\wp_n,
\end{equation}
as desired.

\vspace{8mm}

Before moving on, we determine the quantities $d, \lambda, \varepsilon$. We summarize their required properties, as $n \rightarrow \infty$:
\begin{itemize}
\item $d \rightarrow \infty$ is an integer at most $n/\log n$, and $\varepsilon \rightarrow 0$, $\lambda \rightarrow 0$.
\item $V/2 = d\lambda \rightarrow \infty$.
\item $m = \varepsilon / \lambda \rightarrow \infty$ is an integer.
\item $d\log2 - cn < 0$.
\item $d\lambda^2 \rightarrow 0$.
\end{itemize}

It is enough to specify the growth rate of each variable. If we set $d \approx n^{0.99}, \lambda \approx n^{-0.5}$, and $\varepsilon \approx n^{-0.1}$, say, all the above conditions are fulfilled.

\vspace{8mm}

We prove that \eqref{eq_goal_rere} is $o_n(1)$. Write $N = V/2$ and, without loss of generality, assume $N$ is an integer. Let $X_1, X_2, \ldots$ be independent exponential distributions of rate $1/2$. Instead of integrating over $D(\mathcal{R}^n)$, we wish to integrate over $\mathbb{R}^N_{>0}$, and with respect to the measure imposed by $X_1, \ldots, X_N$ instead of $\wp_n$. To be precise, the goal is to reduce the estimation of \eqref{eq_goal_rere} to the estimation of
\begin{equation} \label{eq_goal_final}
\int_{\mathbb{R}^N_{>0}} \left| \int_0^\infty \frac{1}{N}\sum_{i=1}^N\delta(s-X_i)h_n(s)ds - \int_0^\infty \frac{1}{2}e^{-\frac{1}{2}s}h_n(s)ds \right| dX_1 \ldots dX_N.
\end{equation}

To this end, we will show first that
\begin{equation*}
\int_{D(\mathcal{R}^n)} \left|\int_0^\infty h_n(s)\left( P(s,N(A),w) - P(s,N,w)\right) ds\right| d\wp_n = O(N^{-1/3}),
\end{equation*}
where
\begin{equation*} \label{eq_compare}
P(s,N(A),w) = \frac{1}{N(A)} \sum_{j: x_j \in A} \delta(s-X_j(w))
\end{equation*}
as earlier (here $N(A) = |w \cap A/\{\pm 1\}|$ depends on $w$ as said earlier, and $X_j(w) := \nu_{j+1}(w) - \nu_j(w)$), and
\begin{equation*}
P(s,N,w) := \frac{1}{N} \sum_{j = \min\{i:x_i \in A\}}^{N+\min\{i:x_i \in A\}-1} \delta(s-X_j(w)).
\end{equation*}

By Chebyshev, $|N - N(A)| < N^{2/3}$ except on a set $E$ of $\wp_n$-measure $O(N^{-1/3})$. Since the integrand is bounded (again, because $h_n$ is supported away from zero), on $E$ the integral is $O(N^{-1/3})$ as desired. For each $w \not\in E$,
\begin{equation*}
P(s,N(A),w) - P(s,N,w) = \frac{N-N(A)}{N \cdot N(A)}\left(\mathrm{at\ most\ } N^{2/3}\ \mathrm{terms}\right),
\end{equation*}
so again the integral is $O(N^{-1/3})$ on $D(\mathcal{R}^n) \backslash E$. It is clear that in both cases the implied constants depend only on $h$.

Therefore we can replace $P(s,N(A),w)$ in \eqref{eq_goal_rere} with $P(s,N,w)$, which then becomes
\begin{equation*}
\int_{D(\mathcal{R}^n)} \left| \int_0^\infty \frac{1}{N} \sum_{j = \min\{i:x_i \in A\}}^{N+\min\{i:x_i \in A\}-1}\delta(s-X_j(w))h_n(s)ds - \int_0^\infty \frac{1}{2}e^{-\frac{1}{2}s}h_n(s)ds \right| d\wp_n.
\end{equation*}
But this equals \eqref{eq_goal_final}, because $X_j(w)$ are indeed independent exponential distributions of rate $1/2$.

%So for each $k \in \mathbb{Z}_{\geq 0}$, we must compare
%\begin{equation} \label{eq_compare}
%\frac{1}{N(A)}\sum_{i=1}^{N(A)} I(\nu_i \in [k\varepsilon, (k+1)\varepsilon)) \mbox{ to } \frac{1}{N}\sum_{i=1}^{N} I(\nu_i \in [k\varepsilon, (k+1)\varepsilon)).
%\end{equation}

%Of course the variables are dependent on the implicit $w \in D(\mathcal{R}^n)$. By our usual application of the Chebyshev inequality, it suffices to consider those $w$'s for which $|N - N(A)| < N^{2/3}$. One can easily compute that the two expressions in \eqref{eq_compare} differ by $O(N^{-1/3})$. Summing up over all $k$ means that we must multiply the error by $O(\varepsilon^{-1})$. Even then the error vanishes with our prior settings of $d, \lambda, \varepsilon$.

Finally it remains to prove \eqref{eq_goal_final}. It can be done by invoking the central limit theorem and related machinery; here we give a low-tech alternative argument. Let $K = \sup\{x: x \in \mathrm{supp}\ h_n\}$, $p_k = \int_{k\varepsilon}^{(k+1)\varepsilon} \frac{1}{2}e^{-\frac{1}{2}s} ds$, $q_k = \frac{1}{N}\left|\{i:X_i \in [k\varepsilon, (k+1)\varepsilon)\}\right|$. Note that $Nq_k$ has the binomial distribution $B(N,p_k)$. Rewrite \eqref{eq_goal_final} into
\begin{equation*}
\int_{\mathbb{R}^N_{>0}} \left| \sum_{k=0}^{K/\varepsilon} h_n(k\varepsilon)\left(q_k - p_k \right) \right| dX_1 \ldots dX_N.
\end{equation*}

By Chebyshev's inequality again, for each $k$, $\mathrm{Prob}(N|q_k - p_k| > (Np_k)^{2/3}) < (1-p_k)(Np_k)^{-1/3} < (Np_k)^{-1/3}$. Since $p_k = \frac{1}{2}e^{-k\varepsilon/2}\varepsilon + O(\varepsilon^2)$ for all $0 \leq k \leq K/\varepsilon$, we have $p_k^{-1} = O(\varepsilon^{-1})$, and thus $\mathrm{Prob}(N|q_k - p_k| > (Np_k)^{2/3} \mathrm{\ for\ some\ } 0 \leq k \leq K/\varepsilon)$ is bounded from above by
\begin{equation*}
O((N\varepsilon)^{-1/3}) \cdot \frac{K}{\varepsilon} = O\left({d^{-1/3}\lambda^{-1/3}\varepsilon^{-4/3}}\right),
\end{equation*}
which is $O(n^{-0.03})$ with our prior determination of $d, \lambda, \varepsilon$. On the other hand, on the set satisfying $N|q_k - p_k| \leq (Np_k)^{2/3}$ for all $k$,
\begin{equation*}
\left|\sum_{k=0}^{K/\varepsilon} h_n(k\varepsilon)\left(q_k - p_k \right)\right| \leq \|h\|_\infty \cdot \frac{p_k^{2/3}}{N^{1/3}} \cdot \frac{K}{\varepsilon} = O(N^{-1/3}\varepsilon^{-1/3}),
\end{equation*}
which is also $o_n(1)$. This proves that \eqref{eq_goal_final} is $o_n(1)$, completing the proof of (1.3).

\end{document}